%% file: main.tex
\documentclass{article}
\usepackage[utf8]{inputenc}
\usepackage{amsfonts}
\usepackage{amsmath}
\usepackage{amssymb}
\usepackage{amsthm}
\usepackage[
backend=biber,
style=alphabetic,
]{biblatex}
\usepackage{geometry}
\usepackage{lipsum}
\newcommand\blfootnote[1]{%
  \begingroup
  \renewcommand\thefootnote{}\footnote{#1}%
  \addtocounter{footnote}{-1}%
  \endgroup
}

\newtheorem*{rem}{Remark}
\newtheorem{prop}{Proposition}
\newtheorem{lem}[prop]{Lemma}
\theoremstyle{definition}\newtheorem{defn}[prop]{Definition}
\newtheorem{thm}[prop]{Theorem}
\theoremstyle{definition} 
\newcommand{\sslash}{\slash \!\!\slash}
\newcommand{\Spec}{\mathrm{Spec}\,}

\bibliography{references.bib}
\title{The finite generation ideal for Daigle \& Freudenburg's counterexample to Hilbert's fourteenth problem}
\author{Simon Hart, University of York, UK, s.hart@york.ac.uk}

\begin{document}

\maketitle

\begin{abstract}
    We compute the finite generation ideal for Daigle and Freudenburg's counterexample to Hilbert's fourteenth problem. This ideal helps to understand how far the ring of invariants is from being finitely generated.  
    Our calculations show that the finite generation ideal is the radical of an ideal generated by three infinite families of invariants. 
    We show that these three families together with an additional invariant form a SAGBI-basis. 
    We use the properties of our SAGBI-basis in our computation of the finite generation ideal. 
    
\end{abstract}
\section{Introduction}
Let $\mathbb{K}$ be a field, and let $\mathbb{K}[x_1,x_2,\dots,x_n]$ be a polynomial ring in $n$ variables over $\mathbb{K}$, with $\mathbb{K}(x_1,x_2,\dots,x_n)$ its field of fractions and $L$ a subfield of $\mathbb{K}(x_1,x_2,\dots,x_n)$. In his fourteenth problem, Hilbert asked whether the subalgebra $L \cap \mathbb{K}[x_1,x_2,\dots,x_n]$ is finitely generated. In characteristic zero, this has been shown to not always be the case, with Nagata finding the first counterexample in 1959 \cite{NAG}. Further examples have been found, for example by Roberts in 1990, \cite{ROBS}. In 1994, \citeauthor{ACA} showed that Roberts' example arises as the invariant ring of a $\mathbb{G}_a$-action on $\mathbb{A}^7$  in \cite{ACA}. The smallest known counterexample to Hilbert's problem which arises as an invariant ring was found by Daigle and Freudenburg in dimension five, \cite{DAI}. Daigle and Freudenburg's example arises as an invariant ring of the following $\mathbb{G}_a$-action on $\mathbb{A}^5$ in characteristic zero:
\[ \alpha \cdot (a,b,c,d,e) = \left(a, b + \alpha a^3,c + \alpha b + \frac{1}{2} \alpha^2 a^3, d + \alpha c + \frac{1}{2}\alpha^2 b + \frac{1}{6}\alpha^3 a^3, e + \alpha a^2\right).
\]
The invariant ring of an additive group action is known to correspond to the kernel of a locally nilpotent derivation, see for example \cite[\S 1.5]{FRE}. Daigle and Freudenburg's $\mathbb{G}_a$-action corresponds to the kernel of the locally nilpotent derivation 
\[ D:= x^3 \frac{\partial}{\partial s} + s \frac{\partial}{\partial t} + t \frac{\partial}{\partial u} + x^2 \frac{\partial}{\partial v}, \] 
on the polynomial ring $R:= \mathbb{K}[x,s,t,u,v]$. 
There are many algorithms which have been developed to aid in computing kernels of locally nilpotent derivations which we exploit in this paper. Daigle and Freudenburg's example is also closely related to Roberts' example in dimension seven, and a counterexample found by Freudenburg in dimension six, \cite{FREU}. It is possible to construct Daigle and Freudenburg's counterexample  by ``removing symmetries'' from \citeauthor{ROBS}' example, \cite[\S 7.2]{FRE} and, there is a $\mathbb{K}$-algebra homomorphism from Freudenburg's example to Daigle and Freudenburg's which induces a surjective homomorphism on the invariant rings, \cite[\S 2]{TAN}. 

\citeauthor{CAS} provide an infinite family of non-finitely generated $\mathbb{K}$-algebras, which \citeauthor{DOR} realised as the ring of invariants of a $\mathbb{G}_a$-action on a polynomial ring. Additionally, \citeauthor{KUR} has generalised Robert's example in \cite{KUR}. However, it remains a difficult task to construct counterexamples as invariant rings from $\mathbb{G}_a$-actions; and little is known about the structure of these invariant rings in general.



In this paper we determine the \emph{finite generation ideal} of the invariant ring, $R^D$, defined below. That is, the radical ideal of elements $f \in R^D$ for which $R^D_f$ is finitely generated, 
\cite[\S 2]{Kem}. Such an ideal can be understood to track how far a ring is from being finitely generated, with it being the ring itself when $R^D$ is finitely generated.
\citeauthor{DUF} computed the finite generation ideal for Roberts' example in \cite[\S 9]{DUF} and our computation shows that the finite generation ideal is what would be expected by ``removing symmetries" from Roberts' example. Preliminaries and some early results on Daigle and Freudenberg's example are covered in \S 2 and \S 3 respectively.

In order to compute the finite generation ideal, we first construct a generating set for the invariant ring in \S 4 with useful properties. This requires us to construct three infinite families of invariants using a method similar to van den Essen in \cite{VAN}. We then show that these families, together with an additional invariant, generate the invariant ring. In fact, we show that this generating set forms a SAGBI-basis for $R^D$; that is, a Subalgebra Analogue for a Gr\"obner Basis of Ideals, Definition \ref{SAGBI-def}. Our calculation of the SAGBI-basis uses an argument similar to that of \citeauthor{KUR} for Roberts' example, \cite[\S 3]{KUR}. SAGBI-bases were first constructed by Robbiano and Sweedler in \cite{ROB} and Kapur and Madlener \cite{KAP} independently. The properties of a SAGBI-basis and the relations between its elements are key to our computation of the finite generation ideal in \S 5, which comprises Theorem \ref{finite generation ideal}, our main result. We additionally show that the leading terms of these three infinite families generate the subalgebra generated by the leading terms of the finite generation ideal. 

\section{Preliminaries}
Throughout the following we fix $\mathbb{K}$ to be an algebraically closed field of characteristic zero. We will begin with a few preliminaries, followed by a discussion of the invariant ring of a group action. We follow parts of \cite{ALG}, \cite{DUF} and \cite{FRE}. Let $R$ be a commutative $\mathbb{K}$-domain and recall that a derivation $D\colon R \rightarrow R$ is 
\emph{locally nilpotent} if, for all $a \in R$, there is some $n \in \mathbb{N}$ for which $D^n(a) = 0$. We denote the kernel of the derivation by $R^D$ and its image by $D(R)$. Note that $R^D$ is a subring of $R$. We call an element
$p \in R$ with 
$D(p) \in \, R^D$ and $D(p) \neq 0$ a \emph{local slice} for $D$. The \emph{plinth ideal} is defined as $ \mathfrak{pl}(D):=R^D \cap D(R)$; it is simple to check that $\mathfrak{pl}(D)$ is an ideal of $R^D$, see for example \cite[p.17]{FRE}.
\begin{defn}
Let $R$ be a $\mathbb{K}$ domain, the \emph{finite generation ideal} of $R$ is defined as 
\[ \mathfrak{f}_{R}:= \lbrace g \in R \setminus \lbrace 0 \rbrace \, \vert \, R_g \text{ is finitely generated as a }\mathbb{K}\text{-algebra}\, \rbrace \cup \lbrace 0 \rbrace.  \] 
\end{defn}
Note that if $R$ is finitely generated, then $\mathfrak{f}_{R} = R$. If $R$ is a subalgebra of a finitely generated algebra, then $\mathfrak{f}_R$ is non-zero; additionally, $\mathfrak{f}_R$ is a radical ideal, \cite[\S 2.2]{Kem}. One can hence view the finite generation ideal as a form of measure of how far $R$ is from being finitely generated by comparing $\mathfrak{f}_R$ to $R$. Of particular interest to us in this paper is $\mathfrak{f}_{R^D}$, where $D$ is a locally nilpotent derivation. 

If $R= \mathbb{K}[X_1,\dots, X_n]$ we call an element of the form $X_1^{a_1}\cdots X_n^{a_n}$ a \emph{monomial} and an element of the form $\alpha \cdot X_1^{a_1}\cdots X_n^{a_n}$, where $\alpha \in \mathbb{K} \setminus \lbrace 0 \rbrace$ is called a \emph{term}; 
a \emph{polynomial} is a sum of terms. 
\begin{defn}
Suppose $R = \mathbb{K}[X_1,X_2,\dots,X_n]$ and suppose that $D = a_1\frac{\partial}{\partial X_1} + \dots + a_n \frac{\partial}{\partial X_n}$ is a derivation on $R$. We say that $D$ is a \emph{monomial} derivation if each $a_i \in \mathbb{K}[X_1,X_2,\dots,X_n]$ is a monomial; a derivation is called \emph{triangular} if $a_1 \in \mathbb{K}$ and each $a_i \in \mathbb{K}[X_{1},\dots,X_{i-1}]$ for $2 \leq i \leq n$.
\end{defn}

If $R$ is a $\mathbb{K}$-domain then $V = \Spec(R)$ denotes the corresponding scheme, which is an affine variety when $R$ is affine. We denote affine $k$-space by $\mathbb{A}^k$, with coordinate ring the polynomial ring $\mathbb{K}[x_1, \dots, x_k]$. Suppose that $G$ is an algebraic group acting on an affine variety $V$, then $G$ acts by $\mathbb{K}$-algebra automorphisms on $R$ as
\[g\cdot f(v) := f(g^{-1}\cdot v)  \text{ for all } x \in V, \, f \in R.\]
The \emph{invariant ring} for this action is
\[ R^G:= \lbrace f \in R \, \vert \, g \cdot f = f \text{ for all } g \in G \rbrace. \]
An element $f \in R^G$ is called an \emph{invariant}. The \emph{fixed point set} of the action is defined as
\[ V^G:= \lbrace x \in V \, \vert \, g \cdot x = x \text{ for all } g \in G \rbrace. \]

We focus on the case where $G = \mathbb{G}_a = (\mathbb{K}, +)$ is the additive group of the field $\mathbb{K}$. It is known that in this case we have a one-to-one correspondence between algebraic group actions and locally nilpotent derivations, \cite[\S 1.5]{FRE}. Let $D$ be a locally nilpotent derivation on $R$, an element $\alpha \in \mathbb{G}_a$ acts on $R$ as
\[ \exp(\alpha D)(f) := \sum_{i=0}^\infty \frac{1}{i!} \alpha^i D^i(f). \]
Note that since $D$ is locally nilpotent, for each $f \in R$ the displayed sum has only finitely many non-zero terms. Conversely a given $\mathbb{G}_a$-action $\rho:\mathbb{G}_a \times V \rightarrow V$ induces a derivation $\rho'(0)$, which can be shown to be locally nilpotent, see for example \cite[\S 1.5]{FRE}. Additionally, the invariant ring of the group action is equal to the kernel of the derivation, that is $R^{\mathbb{G}_a} = R^D$. 

Given an affine variety $V$ and an algebraic group $G$ acting on $V$, we wish to define a quotient $V \sslash G$. When $G$ is reductive, the ring $R^G$ is finitely-generated, so we get a corresponding variety $V \sslash G:=\Spec(R^G)$. When $G$ is not reductive, $R^G$ may not be finitely generated. Nevertheless, we may still define $V\sslash G:=\Spec(R^G)$ as an affine scheme, and the usual universal property still holds in the category of affine schemes, \cite[p.3]{MUMFORD}. We record this in the following definition:

\begin{defn}
Given $V$, an affine variety, and $G$, an algebraic group acting on $V$, there is a morphism induced by the inclusion $R^G \subset R$:
\[ \pi_V\colon V \rightarrow V \sslash G:= \Spec(R^G), \]
we call this the \emph{quotient morphism}. $V \sslash G$ is the \emph{categorical quotient} in the category of affine schemes, satisfying the universal property that every $G$-invariant morphism from $V$ to some affine scheme $W$ factors uniquely through $\pi_V$. 
\end{defn}
Given an additive group action of $\mathbb{G}_a$ on $V$, we can ask if there is some $\mathbb{G}_m$-action on $V$ commuting with our $\mathbb{G}_a$-action. If so, then $\mathbb{G}_m$ acts on $V \sslash \mathbb{G}_a$ and hence induces a grading on $R^{\mathbb{G}_a}$, as well as on $R$, see for example \cite[\S 10.2]{FRE}. When $V \subset \mathbb{A}^k$ is affine, there is some maximal subtorus of the natural $k$-dimensional torus action on $\mathbb{A}^{k}$ that is $\mathbb{G}_a$-equivariant.

If $X$ is a variety with an action of the additive group $\mathbb{G}_a$, we say that $X$ is a \emph{trivial $\mathbb{G}_a$-bundle} if there is a $\mathbb{G}_a$-equivariant morphism $X \rightarrow \mathbb{G}_a$. In this case we can identify $X\sslash \mathbb{G}_a$ with $X/\mathbb{G}_a$ and the quotient morphism $\pi_X: X \rightarrow X/\mathbb{G}_a$ admits a section. If $X$ is affine, then $X/\mathbb{G}_a = \Spec(R)^{\mathbb{G}_a}$.

Suppose $V = \Spec(R)$ is an affine $\mathbb{K}$-variety, with $G$ an algebraic group acting on $V$. Suppose also that $R = \bigoplus_{i \geq 0} R_i$ is a graded ring with $R_0 = \mathbb{K}$, and let $z_0$ be its homogeneous maximal ideal. This means that $V$ admits an action of the multiplicative group $\mathbb{G}_m = \mathbb{K}^*$ with $z_0 \in \Spec(R)$ the unique closed orbit. We say that $V$ is a \emph{fix-pointed} $G$-variety with fixed point $z_0$ if this $\mathbb{G}_m$-action commutes with the $G$-action. Note that $V \sslash G$ is also fix-pointed, and we define the \emph{nullcone} as $\mathcal{N}_V = \pi^{-1}(\pi(z_0))$. 
Thus, given an additive group action of $\mathbb{G}_a$ on an affine variety $V$, the nullcone can be defined using the induced $\mathbb{Z}^r$-grading on $R$.

Let $D$ be a locally nilpotent derivation on an affine $\mathbb{K}$-domain $R$, with $V= \Spec(R)$. Suppose $x \in V$ is a fixed point under the corresponding $\mathbb{G}_a$-action and, for $f \in R$, let $n \in \mathbb{N}$ be such that $D^n(f) \neq 0$ and $D^{n+1}(f) =0$. Then we have
\[ f(x) = \exp (\alpha D)(f)(x) = f(x) + \alpha D(f)(x) + \cdots \frac{1}{n!}\alpha^n D^n(f)(x). \] 
Letting $\alpha$ vary, we conclude that $D^i(f)(x) = 0$ for all $i \geq 1$, and in particular $D(f)(x) = 0$. Now suppose that for $x \in R$ we have $D(f)(x) = 0$ for all $f \in R$. Then for $f \in R$
\[ \exp(\alpha D)(f)(x) = f(x) + \alpha D(f)(x) + \cdots \frac{1}{n!} \alpha^n D^n(f)(x) = f(x), \]
thus $x$ is a fixed point. We have shown that:
\[ V^{\mathbb{G}_a} = \lbrace x \in V \, \vert \, D(f)(x) = 0 \text{ for all } f \in R \rbrace = \lbrace x \in V \, \vert \, f(x) = 0 \text{ for all } f \in D(R) \rbrace. \]
Consider $\mathcal{P}_V:=\mathcal{V}(\mathfrak{pl}(D))$, the \emph{Plinth variety} of $V$. We have 
\[ \mathcal{P}_V = \lbrace x \in V \, \vert \, D(f)(x) = 0 \text{ for all } f \in R \text{ with } D^2(f) = 0 \rbrace, \]
and hence we have $V^{\mathbb{G}_a} \subset \mathcal{P}_V$.
Given a $\mathbb{Z}_{\geq 0}^r$-grading on $R^{\mathbb{G}_a}$ induced by a $(\mathbb{G}_m)^r$-action, we define $R^{\mathbb{G}_a}_+ \! := \lbrace f \in R^{\mathbb{G}_a} \, \vert \, \mathrm{deg}\,(f) \neq (0,0,\dots,0) \rbrace$, the maximal graded ideal of $R^{\mathbb{G}_a}$. Our definition of the nullcone may then be rewritten as 
\[ \mathcal{N}_V = \lbrace x \in V \, \vert \, f(x) = 0 \text{ for all } f \in R^{\mathbb{G}_a}_+ \rbrace. \]

\section{Daigle \& Freudenburg's counterexample}
We now construct Daigle and Fruedenberg's counterexample.
Let $V = \mathbb{A}^5$ and let $R:= \mathbb{K}[x,s,t,u,v]$ be the polynomial ring over $\mathbb{K}$ in $5$ variables. We consider the following locally nilpotent derivation on $R$
\begin{equation} \label{D}
    D:= x^3 \frac{\partial}{\partial s} + s \frac{\partial}{\partial t} + t \frac{\partial}{\partial u} + x^2 \frac{\partial}{\partial v}.
\end{equation} 
This corresponds to a $\mathbb{G}_a$-action on $R$ defined by
\[ \exp(\alpha D) \cdot (x,s,t,u,v) = (x, s + \alpha x^3,t + \alpha s + \frac{1}{2} \alpha^2 x^3, u + \alpha t + \frac{1}{2}\alpha^2 s + \frac{1}{6}\alpha^3 x^3, v + \alpha x^2).  \] 
The $\mathbb{G}_a$-action on $\mathbb{A}^5$ commutes with the following $\mathbb{G}_m$-action 
\begin{equation} \label{G_m action}
    \lambda \cdot (x,s,t,u,v):= (\lambda \cdot x, \lambda^3 \cdot s, \lambda^3 \cdot t, \lambda^3 \cdot u, \lambda^2 \cdot v),
\end{equation} 
which induces a grading on $R$ with deg$(x) =1$, deg$(s)=\,$deg$(t)=\,$deg$(u)=3$ and deg$(v)=2$. In the sequel when we refer to $f \in R$ as \emph{homogeneous}, we mean homogeneous with respect to this grading. Likewise, for $f \in R$, the degree of $f$ is the maximal degree of some term of $f$ with respect to this grading. In our treatment of this example we will occasionally consider elements ordered by deg$_v$ or deg$_x$, which are defined for $f \in R$ as deg$_x(f):=$max$\lbrace n \, \vert \, f \text{ has a term of the form }\alpha \cdot x^ns^at^bu^cv^d, \alpha \in \mathbb{K} \setminus \lbrace 0 \rbrace, a,b,c,d \in \mathbb{N}\rbrace$ and similarly for $v$. A polynomial $f \in R$ has $v$-degree $n$ if $\deg_v(f) = n$.

We also make use of the grading induced by the derivation $D$ itself, which we refer to as the \emph{$\rho$-grading}. It is defined first on the monomials in $R$ with 
\begin{equation} \label{rho-grading}
    \rho(m) := \lbrace i \in \mathbb{Z}_{\geq 0} \,\vert \, D^i(m) \neq 0, D^{i+1}(m) =0 \rbrace.
\end{equation}
We then set, for $f \in R$, $\rho(f) := \max\lbrace \rho(m) \,\vert\, m \text{ is a term of } f \rbrace$. We set $\rho(0):=-\infty$. Elements homogeneous with respect to this grading will be called $\rho$-\emph{homogeneous}. Observe that $\rho(x^a) = 0, \, \rho(s^b) = b, \, \rho(t^c) = 2c, \, \, \rho(u^d) = 3d$ and $\rho(v^e)=e$. Note that the $\rho$-grading is indeed a grading; set 
\[R_n:= \left \lbrace \sum_i \lambda_i x^a s^b t^c u^d v^e \, \vert \, \lambda_i \in \mathbb{K}, a \in \mathbb{N}, b + 2c + 3d + e = n \right\rbrace.  \]
Now, for $p \in R_i, \, \, q \in R_j$, their product is non-zero, and all terms in $pq$ are of the form $mn$, where $m \in R_i, n \in  R_j$. We have 
\[ D^{i+j}(mn) = \sum_{l=0}^{i+j} \binom{i+j}{l} D^l(m) D^{i+j-l}(n) = \binom{i + j}{i} D^i(m) D^j(n) \neq 0,\] whilst $D^{i+j+1}(mn) = D(D^i(m) D^j(n)) = D^{i+1}(m)D^j(n) + D^i(m)D^{j+1}(n) = 0$, so we conclude $pq \in R_{i+j}$. 
\begin{rem}
$\rho(2x^3t-s^2) = 2$ whilst $D(2x^3t-s^2) = 0$, so for $p \in R$, $\rho(p)$ can differ from the unique non-negative integer $m$ with $D^m(p) \neq 0$ but $D^{m+1}(p)=0$. 
\end{rem} 

Now let $S:= \mathbb{K}[x,s,t,u]$, and define 
\[ \Delta:= D \vert_S = x^3 \frac{\partial}{\partial s} + s \frac{\partial}{\partial t} + t \frac{\partial}{\partial u}. \]
Our notions of degree and $\rho$-degree restrict to $f \in S$. We observe that $\Delta$ is a triangular monomial derivation on $\mathbb{K}[x,s,t,u]$. By a result of \citeauthor{MAU}, \cite[\S~3]{MAU}, we have that $S^{\Delta}$ is generated by at most four elements. Through an application of van den Essen's algorithm, \cite[\S~4]{ALG}, using the local slice $s \in S$ we find the following four generators of $S^{\Delta}$:
\begin{equation} \begin{split}\label{S^Delta generators}
 \beta_0  &= x,  \\
 \gamma_0  &= 2x^3t - s^2, \\
 \delta_0  &= 3x^6u - 3x^3st + s^3, \\
 g \, \, &= 9x^6u^2 - 18x^3stu + 6s^3u + 8x^3t^3 - 3s^2t^2.
\end{split} \end{equation}
Observe that $\beta_0^{3}, \gamma_0, \delta_0 \in \Delta(S)$, since
\[ \begin{array}{ll}
\Delta(s)&= x^3 = \beta_0^3,  \\
\Delta(3x^3u - st)&= 2x^3t -s^2 = \gamma_0, \\
\Delta(3x^3su - 4x^3t^2 + s^2t)&= 3x^6u - 3x^3st + s^3 = \delta_0. 
\end{array} \]
We can compute the plinth variety, the fixed-point set and the nullcone for this counterexample:
\begin{lem}
    \begin{enumerate}
        \item $\left(\mathbb{A}^5 \right)^{\mathbb{G}_a} = \mathcal{V}_{\mathbb{A}^5}(x,s,t)$.
        \item $\mathcal{P}_{\mathbb{A}^5} = \mathcal{V}_{\mathbb{A}^5}(x,s).$
        \item $\mathcal{N}_{\mathbb{A}^5}= \mathcal{V}_{\mathbb{A}^5}(x,s).$
    \end{enumerate}
\end{lem}
\begin{proof}
$\emph{1}.$ 
As shown above, we have
\[ \left(\mathbb{A}^5\right)^{\mathbb{G}_a} = \lbrace p \in \mathbb{A}^5 \, \vert \, D(f)(p) = 0 \text{ for all } f \in R \rbrace. \]
Suppose $p = (p_1,p_2,p_3,p_4,p_5) \in (\mathbb{A}^5)^{\mathbb{G}_a}$ and $f \in R$, then
\[ D(f)(p) = p_1^3 \frac{\partial f}{\partial s}(p) + p_2 \frac{\partial f}{\partial t}(p) + p_3 \frac{\partial f}{\partial u}(p) + p_1^2 \frac{\partial f}{\partial v}(p). \]
Clearly if $p = (0,0,0,p_4,p_5)$, then $D(f)(p) = 0$ for all $f \in R$. Conversely suppose at least one of $p_1, p_2$ or $p_3 \neq 0$, then one of $D(s)(p), D(t)(p), D(u)(p)$ is non-zero, so $\left(\mathbb{A}^5 \right)^{\mathbb{G}_a} = \mathcal{V}_{\mathbb{A}^5}(x,s,t)$ as claimed.

\bigskip \noindent $\emph{2}.$ 
By our observations above we have $\beta_0^3, \gamma_0, \delta_0 \in \mathfrak{pl}(D)$, so $\mathcal{V}_{\mathbb{A}^5}(\beta_0^3,\gamma_0,\delta_0) = \mathcal{V}_{\mathbb{A}^5}(x,s) \subset \mathcal{P}_{\mathbb{A}^5}$. 
Let $f \in \mathfrak{pl}(D)$, we will show that $f \in \mathcal{V}_{\mathbb{A}^5}(x,s)$. Suppose that $f$ is homogeneous, we induct on the degree, $n$, induced by the $\mathbb{G}_m$-action introduced in equation \ref{G_m action}. Note that there are no elements of degree $0$ or $1$ in the plinth ideal. The only elements in the plinth ideal of degree $2$ or $3$ are the monomials $x^2$ and $x^3$. 
Now suppose $f \in \mathfrak{pl}(D)$ has degree $n>3$ and that $f = xp_1 + sp_2 + g$, where $p_1,p_2 \in \mathbb{K}[x,s,t,u,v]$ and $g \in \mathbb{K}[t,u,v]$. Observe that the partial derivatives $\frac{\partial}{\partial t}, \frac{\partial}{\partial u}, \frac{\partial}{\partial v}$ all commute with the derivation $D$,
and so for any $p \in R^D$, we have
\[ 0 = \frac{\partial}{\partial t}(D(p)) = D\left( \frac{\partial}{\partial t}(p)\right) = \frac{\partial}{\partial u}(D(p)) = D\left( \frac{\partial}{\partial u}(p)\right) = \frac{\partial}{\partial v}(D(p)) = D\left( \frac{\partial}{\partial v}(p)\right). \]
Hence $\frac{\partial f}{\partial t}, \frac{\partial f}{\partial u}, \frac{\partial f}{\partial v} \in R^D$ since $f \in R^D$. But note that these partial derivatives all have degree $n-3$, hence by induction we have that there are $h_1, h_2, h_3, h_4, h_5, h_6 \in \mathbb{K}[x,s,t,u,v]$ with 
\begin{equation} \label{h1-h6 induction}
    \frac{\partial f}{\partial t} = xh_1 + sh_2, \quad \frac{\partial f}{\partial u} = xh_3 + sh_4, \quad \frac{\partial f}{\partial v} = xh_5 + s h_6.
\end{equation} 
However
\[ \frac{\partial f}{\partial t} = x\frac{\partial p_1}{\partial t} + s \frac{\partial p_2}{\partial t} + \frac{\partial g}{\partial t}, \qquad \frac{\partial f}{\partial u}= x\frac{\partial p_1}{\partial u} + s \frac{\partial p_2}{\partial u} + \frac{\partial g}{\partial u}, \qquad \frac{\partial f}{\partial v}= x\frac{\partial p_1}{\partial v} + s \frac{\partial p_2}{\partial v} + \frac{\partial g}{\partial v}, \]
where $\frac{\partial g}{\partial t}, \frac{\partial g}{\partial u}, \frac{\partial g}{\partial v} \in \mathbb{K}[t,u,v]$, hence these partial derivatives of $g$ must all be zero by equation \ref{h1-h6 induction}. Since we have assumed $f$ is homogeneous, this implies that $g=0$.

\bigskip \noindent $\emph{3}.$ 
Recall that the nullcone is given by 
\[\pi_{\mathbb{A}^5}^{-1}(\pi_{\mathbb{A}^5}(0)) = \mathcal{N}_{\mathbb{A}^5} = \left\lbrace v \in \mathbb{A}^5 \, \vert \, f(v) =0 \text{ for all } f \in R^{\mathbb{G}_a}_+ \right\rbrace.\]
Suppose $v = (v_1,v_2,v_3,v_4,v_5) \in \mathcal{N}_{\mathbb{A}^5}$, 
then for $\beta_0, \gamma_0 \in R^{\mathbb{G}_a}_+$, we have $\beta_0(v) = v_1 = 0$, $\gamma_0(v) = 2(0)^3v_3 - v_2^2 = 0$, so $v = (0,0,v_3,v_4,v_5)$ and 
$\mathcal{N}_{\mathbb{A}^5} \subset \mathcal{V}_{\mathbb{A}^5}(x,s)$. Now suppose that $v \in \mathcal{V}_{\mathbb{A}^5}(x,s)$, and let $f \in R^{\mathbb{G}_a}_+$, the above calculation for the plinth variety shows that $R^{\mathbb{G}_a}_+ \subset (x,s)\mathbb{K}[x,s,t,u,v]$. Therefore we can write $f = xh_1 + sh_2$ with $h_1,h_2 \in \mathbb{K}[x,s,t,u,v]$ and so $f(v)=0$, implying $v \in \mathcal{N}_{\mathbb{A}^5}$ 
and hence that $\mathcal{N}_{\mathbb{A}^5}= \mathcal{V}_{\mathbb{A}^5}(x,s)$.
\end{proof}
\section{A SAGBI-basis for Daigle and Freudenburg's counterexample}
\subsection{Constructing three infinite families of invariants}
We first show that $R^D$ is not finitely generated, to do this we show that $R^D$ has three infinite families of homogeneous invariants. We call the members of these families $\beta_i, \gamma_i$ and $\delta_i$ respectively, with $i \in \mathbb{N}$ corresponding to the $v$-degree of the invariant. Recall that $S^{\Delta}$ is generated $\beta_0, \gamma_0, \delta_0$ and $g$, defined in equation \ref{S^Delta generators}. We construct the $\beta_i, \gamma_i$ and $\delta_i$ so that $\beta_i := \beta_0 v^i + \text{ terms of lower }v\text{-degree}$, and similarly for $\gamma_i$ and $\delta_i$. For $i=1$ we must find, for example, some $f \in R$ so that $D(\beta_0 v + f) = x^3 + D(f) = 0$. This is a simple task for $\beta_1, \gamma_1$ and $\delta_1$ since $x^2\beta_0, x^2\gamma_0, x^2\delta_0 \in \mathfrak{pl}(D)$, giving: 
\[
\begin{array}{l}
\beta_1 = xv - s, \\  
\gamma_1 = (2x^3t - s^2)v + x^2st - 3x^5u, \\ 
\delta_1 = (3x^6u - 3x^3st + s^3)v - 3x^5su +4x^5t^2 - x^2s^2t.
\end{array} \]
In general these invariants are difficult to construct, but we show that such invariants exist. Once this is accomplished, we construct a SAGBI-basis for $R^D$.
\begin{defn}
Let $\mathcal{M}$ be the set of all monomials in $\mathbb{K}[x_1,x_2,\dots,x_n]$, A \emph{monomial} ordering is a total order $``\!>\!"$ on $\mathcal{M}$ which satisfies the following conditions:
\begin{itemize}
    \item $m > 1 \text{ for all } m \in \mathcal{M} \setminus \lbrace 1 \rbrace$, 
    \item $m_1 > m_2$ implies $bm_1 > bm_2$ for all $b,m_1,m_2 \in \mathcal{M}$. 
\end{itemize} 
We write $x_i >\!> x_j$ if $x_i > x_j^a$ for all $a \in \mathbb{Z}_{\geq 0}$. Given a non-zero polynomial $f \in \mathbb{K}[x_1,x_2,\dots,x_n]$, we can write $f$ uniquely as $f=cm+g$, where $m \in M$, $c \in \mathbb{K} \setminus \lbrace 0 \rbrace$ and every monomial appearing as part of a term in $g$ is smaller than $m$ with respect to our ordering. We call $cm$ the \emph{leading term} of $f$, and write LT$(f) = cm$. Additionally, $m$ is the \emph{leading monomial} of $f$, denoted LM$(f)$.
\end{defn}
We may now define a Gr\"obner basis, we use \cite[p.~10]{DER}.
\begin{defn}
Fix a monomial ordering on $\mathbb{K}[x_1,x_2,\dots,x_n]$ and let $S \subset \mathbb{K}[x_1,x_2,\dots,x_n]$ be a set of polynomials. We write 
\[ \text{L}(S) := \left( \text{LM}(f) \, \vert \, f \in S \right), \]
for the ideal generated by the leading monomials from $S$, called the \emph{leading ideal} of $S$. Now let $I\subset\mathbb{K}[x_1,x_2,\dots, x_n]$ be an ideal, then a finite subset $\mathcal{G} \subset I$ is called a Gr\"obner basis for $I$ if L$(I) = \,$L$(\mathcal{G})$. 
\end{defn}
\begin{rem}
A Gr\"obner basis of $I$ generates $I$ as an ideal: suppose that $f \in I$ is an element with minimal leading monomial which is not contained in the ideal generated by $\mathcal{G}$. Then, as $\mathcal{G}$ is a Gr\"obner basis, \emph{LM}$(f) \in \,\, $\emph{L}$(\mathcal{G})$, so there is some $h \in (\mathcal{G})$  with \emph{LM}$(h) = \, $\emph{LM}$(f)$. Now $f - h$ has smaller leading monomial than $f$, and must be contained in the ideal generated by $\mathcal{G}$. This means $f = (f-h) + h$ is contained in the ideal generated by $\mathcal{G}$, a contradiction.  
\end{rem}
An analogous concept for subalgebras also exists, called a SAGBI-basis  
\begin{defn} \label{SAGBI-def}
A \emph{Subalgebra Analogue for Gr\"obner Bases of Ideals} or ``SAGBI-basis" is defined as follows: Let $``\!>\!"$ be a monomial ordering on the polynomial ring $\mathbb{K}[x_1,x_2,\dots,x_n]$. For a subalgebra $A \subset \mathbb{K}[x_1,x_2,\dots,x_n]$, we write $\mathrm{L}_{alg}(A)$ for the algebra generated by all leading monomials of non-zero elements in $A$. A subset $\mathcal{S} \subset A$ is called a SAGBI-basis of $A$ if $\mathrm{L}_{alg}(\mathcal{S}) = \mathrm{L}_{alg}(A)$.
\end{defn}
Note also that a SAGBI-basis $\mathcal{S}$ of a subalgebra $A$ also generates $A$ as an algebra as in the case of a Gr\"obner basis.

In the sequel we use the lexicographic monomial ordering on $R=\mathbb{K}[x,s,t,u,v]$, which is defined so that $x^{e_1}s^{e_2}t^{e_3}u^{e_4}v^{e_5} > x^{f_1}s^{f_2}t^{f_3}u^{f_4}v^{f_5}$ if $e_i > f_i$ for the largest $i$ for which we have $e_i \neq f_i$. For example $tv > v$ since both monomials have $v$ exponent $1$, $u$-exponent $0$ but $tv$ has $t$ exponent $1$ whilst $v$ has $t$ exponent $0$. We also use the lexicographic monomial ordering on $S = \mathbb{K}[x,s,t,u]$ with $x<s<t<u$.
\begin{defn}
Let $\mathcal{S} = \lbrace f_1, \dots, f_m \rbrace \subset \mathbb{K}[x_1,\dots,x_n]$ be a finite set of polynomials.
\begin{enumerate}
    \item A polynomial $p \in \mathbb{K}[x_1,\dots,x_n]$ is said to be in \emph{normal form} with respect to $\mathcal{S}$ if no term of $p$ is divisible by the leading monomial of any $f \in \mathcal{S}$.
    \item If $p, \Tilde{p} \in \mathbb{K}[x_1,\dots,x_n]$, $\Tilde{p}$ is said to be a \emph{normal form} of $p$ with respect to $\mathcal{S}$ if $\Tilde{p}$ is in normal form with respect to $\mathcal{S}$ and there are $h_1,\dots,h_m \in \mathbb{K}[x_1,\dots,x_n]$ with
    \[ p - \Tilde{p} = \sum_{i=1}^m h_if_i \quad \text{ and LM}(h_if_i) \leq \text{LM}(p) \text{ for all } i. \]
\end{enumerate}
\end{defn}

We now state the image membership algorithm, see van den Essen \cite[\S 1.4]{ARN}.
\begin{lem}[Image membership algorithm] \label{Image membership algorithm}
Let $S=\mathbb{K}[x_1,\dots,x_n]$ be a finitely generated $\mathbb{K}$-algebra, and $\mathcal{D}$ a non-zero locally nilpotent derivation on $S$. Fix $a \in S$, and let $m$ be the unique non-negative integer satisfying $\mathcal{D}^m(a) \neq 0$, $\mathcal{D}^{m+1}(a)=0$.
Let $p$ be a local slice of $\mathcal{D}$, with $d:=\mathcal{D}(p)$ and $s:=p/d \in S[d^{-1}]$. Suppose $S^{\mathcal{D}} = \mathbb{K}[f_1,\dots,f_l]$, put
\[ b':= \sum_{i=0}^{m} \frac{(-1)^i}{(i+1)!} \mathcal{D}^i(a) s^{i+1}, \]
and set $q:=d^{m+1}b'$. Define the ideal $J_m$ in $\mathbb{K}[X,Y]:=K[x_1,\dots,x_n,y_1,\dots y_l]$ as: 
\[ J_m := (y_1 - f_1, \dots, y_l-f_l, d^{m+1}), \]
and choose on $\mathbb{K}[X,Y]$ a monomial ordering so that $x_i >\!>  y_j$ for all $i,j$. Let $\mathcal{G}$ be a Gr\"obner basis of $J_m$. Let $\Tilde{q}$ be the normal form of $q$ with respect to $\mathcal{G}$.
Then $a \in \mathcal{D}(S)$ if and only if $\Tilde{q} \in \mathbb{K}[Y]$. Furthermore, if $\Tilde{q} \in \mathbb{K}[Y],$ then $b:= (q - \Tilde{q}(f_i))/d^{m+1} \in S$ satisfies $\mathcal{D}(b) = a$. The polynomial $\Tilde{q}(f_i)$ is defined by replacing each $y_i$ appearing in $\Tilde{q}$ with $f_i$.
\end{lem}
Now we show the existence of the $\beta_i, \gamma_i$ and $\delta_i$:
\begin{prop} \label{Infinite invariants} For each $n \in \mathbb{N}$, there are invariants $\beta_n, \gamma_n, \delta_n \in R^D$ with leading terms $\beta_0v^n, \gamma_0v^n$ and $\delta_0 v^n$.
\end{prop}
A proof of Proposition \ref{Infinite invariants} has been given by \citeauthor{TAN} in \cite[\S 2]{TAN}, making use of the relation between Daigle and Freudenburg's counterexample and Freudenburg's counterexample. Here we provide a direct proof. 
\begin{proof}
\input{Prop1}

\end{proof}
\begin{rem}
    Note that the choice of the $e_0^{(n)}$ is not unique in general. However, given say $b_0^{(n)}$ and some $\tilde{b}_0^{(n)}$, we must have $D(b_0^{(n)} - \tilde{b}_0^{(n)})=0$. As both are chosen to be homogeneous with respect to our gradings, they must differ by an element of $S^{\Delta} \cap S_{(2n+1,n)}$. This vector space is non-empty precisely when $n=6k$, with basis $xg^k$ in this case. Since $b_0^{(0)} = x$, we observe that the $b_0^{(n)}$ are unique for $n \leq 5$. For $c_0^{(n)}$ and $d_0^{(n)}$, we find that these are unique for 
    $n \leq 3$ and $n \leq 2$ respectively.
\end{rem}
\subsection{A generating set for $R^D$}
\begin{lem} \label{S generates A}
The set of invariants 
\[ \mathcal{S}:= \left \lbrace g, \beta_n, \gamma_n, \delta_n \vert \, n \in \mathbb{N} \right \rbrace, \]
generates $R^D$. 
\end{lem}
\begin{proof}
We prove this by induction on $n$, the degree of $v$. Namely we show that the set 
\[ \mathcal{S}_n :=\left \lbrace g, \beta_m, \gamma_m, \delta_m \vert \, m \leq n \right \rbrace, \]
generates $A_n = \lbrace f \in R^D \, \vert \, \text{deg}_v(f) \leq n \rbrace$.

When $n=0$, we are considering the elements of $v$-degree $0$, but these are the invariants in $\mathbb{K}[x,s,t,u] = S$ where $D\vert_S = \Delta$, and $S^{\Delta}$ we know $S^{\Delta}$ is generated by $\beta_0, \gamma_0, \delta_0$ and $g$, which is just $\mathcal{S}_0$. Now suppose that $A_k$ is generated by $\mathcal{S}_k$ for all $k \leq n-1$, and let $f \in R^D$ be an invariant whose terms have $v$-degree at most $n$. Without loss of generality we may assume that $f$ is homogeneous with respect to our $\mathbb{G}_m$-grading and that
\[ f = a_nv^n + a_{n-1}v^{n-1} + \dots + a_0, \]
with $a_i \in K[x,s,t,u]$ for all $i$, and $a_n \neq 0$. Now $D(f) =0$, meaning that 
\begin{align*}
 D(f) & = D(a_nv^n + a_{n-1}v^{n-1} + \dots + a_0) \\
& = D(a_n)v^n + nx^2a_nv^{n-1} + D(a_{n-1}v^{n-1} + \dots + a_0) = 0.
\end{align*}
Now comparing $v$-degrees, we see that we must have $D(a_n) = 0$, but $a_n \in \mathbb{K}[x,s,t,u]$ and so $a_n \in S^{\Delta}$ which is generated by $g, \beta_0, \gamma_0$ and $\delta_0$. So we may write
\[ a_n = \beta_0 p_1 + \gamma_0 p_2 + \delta_0 p_3 + \lambda g^k,   \]
for some $p_1, p_2, p_3 \in \mathbb{K}[\beta_0,\gamma_0,\delta_0,g]$, $\lambda \in \mathbb{K}$ and $k \geq 0$. But if we define 
\[ G = \beta_n p_1 + \gamma_n p_2 + \delta_n p_3, \]
then $D(G)=0$ and $G$ is generated by $\mathcal{S}_n$ as $\beta_n, \gamma_n, \delta_n \in \mathcal{S}_n$ and $p_1,p_2,p_3 \in \mathcal{S}_0 \subset \mathcal{S}_n$. We also have $D(G-f)=0$, where 
\[ G-f = \lambda g^k v^n + b_{n-1}v^{n-1} + \dots + b_0. \]
We show that no such invariant can exist unless $\lambda = 0$, in which case each term of $G-f$ has $v$-degree at most $n-1$ and therefore must be generated by $\mathcal{S}_{n-1}$ by induction. We can then conclude that $f$ is generated by $\mathcal{S}_n$, proving the result.

If $\lambda \neq 0$ then we can take $\lambda = 1$ by re-scaling and our task becomes showing that there is no invariant of the form 
\[  h = g^kv^n + b_{n-1}v^{n-1} + \dots + b_0. \]
To do this we consider 
\begin{align*}
   D(h)  = & D(g^k)v^n + (nx^2g^k + D(b_{n-1}))v^{n-1} + (n-1)x^2b_{n-1}v^{n-2} \\
   & + D(b_{n-2}v^{n-2} + \dots + b_0).
\end{align*}
Considering the terms of $v$-degree $n-1$ we see that, to have $D(h)=0$, we require that $D(b_{n-1}) = \Delta(b_{n-1}) = -nx^2g^k$. In other words, we require that $-nx^2g^k \in \Delta(S)$. We show that this is not the case for all $k \in \mathbb{N}$ by use of Lemma \ref{Image membership algorithm}, the image membership algorithm. By choosing our local slice to be $p=s \in R$, with $d= \Delta(s) = x^3$, we compute the Gr\"obner basis of the ideal
\[ J:=(y_1 -\beta_0,\, y_2 - \gamma_0,\, y_3 -\delta_0,\, y_4 - g,\, x^3), \]
with the lexicographic monomial ordering chosen so that $u>t>s>x>y_4>y_3>y_2>y_1$. Using computational software such as Maple, we are able to find that our Gr\"obner basis is then
\[ \mathcal{G} = ( y_1^3,\, y_2^3 + y_3^2,\, x-y_1,\, sy_2 + y_3,\,sy_3-y_2^2,\, s^2 + y_2, \, 6y_2^2 u -3y_3t^2 -y_4s,\, 6y_3u + 3y_2t^2 -d). \]
Now, since $\Delta(x^2g)=0$, we find that $b' = x^{-1}g^k s$ and hence in the notation of Lemma \ref{Image membership algorithm} we have $q=x^2g^k s$. The normal form of $q$ with respect to this basis is $\tilde{q} = y_1^2y_4^k s \notin \mathbb{K}[y_1,y_2,y_3,y_4]$, therefore by the image membership algorithm, $x^2 g^k \notin \, \Delta(S)$ for all $k \in \mathbb{N}$.
\end{proof}
\subsection{Computing a SAGBI-basis}
Now we will show that the set $\mathcal{S}$ forms a SAGBI-basis for our invariant ring $R^{\mathbb{G}_a}$. We follow a method similar to that used in \cite[\S~3]{KUR}, where a SAGBI-basis for Roberts' counterexample is computed.
We recall from Definition \ref{SAGBI-def} that for a subalgebra $R$, a SAGBI-basis of $R$ is a subset $\mathcal{S} \subset R$ which satisfies $\text{L}_{alg}(R) = \text{L}_{alg}(\mathcal{S})$, where $\text{L}_{alg}(R)$ denotes the algebra generated by the leading monomials of the elements in $R$. Note that for our chosen $\mathcal{S}$ we have
\[ \text{L}_{alg}(\mathcal{S}) = \mathbb{K}[xv^n, x^3tv^n, x^6uv^n, x^6u^2 \, \vert \, n \in \mathbb{N}]. \] 
We set
\[ \begin{array}{c c c c}
b_n := xv^n,  & c_n := 2x^3tv^n, & d_n := 3x^6uv^n, & e := 9x^6u^2. 
\end{array} \]
Also note that, as remarked above, since the $e_0^{(n)}$ can be determined uniquely for $n\leq 2$, for the invariants $\beta_n, \gamma_n, \delta_n$ we can write the terms of $v$-degree $n, n-1$ and $n-2$. Namely, we have:  
\[
\begin{array}{l}
    \beta_n = xv^n - nsv^{n-1} + n(n-1)x^2tv^{n-2} + l.o.t, \\[5pt]
    \gamma_n = (2x^3t - s^2)v^n - n(-3x^5u + x^2st)v^{n-1} + n(n-1)(3x^4su - 2x^4t^2)v^{n-2} + l.o.t, \\[5pt] 
    \delta_n = (3x^6u - 3x^3st + s^3)v^n - n(3x^5su + 4x^5t^2 - x^2s^2t)v^{n-1} - n(n-1)(3x^7tu - 3x^4s^2u + x^4st^2) + l.o.t,
\end{array}\] 
where $l.o.t.$ refers to terms of lower $v$-degree. Recall from the proof of Lemma \ref{S generates A}:
\[ \mathcal{S}_N := \left\lbrace \beta_i, \gamma_i, \delta_i, g \, \vert \, 0 \leq i \leq N \right\rbrace. \] 
Let $B_N$ be the subalgebra of $R^D$ generated by $\mathcal{S}_N$ for all $N \geq 0$.
\begin{lem} \label{SAGBI}
For all $N \geq 0$ the subalgebra $\emph{L}_{alg}(B_N) \subset R$ is generated by $\emph{L}_{alg}(\mathcal{S}_N)$, hence $\mathcal{S}_N$ is a SAGBI-basis of $B_N$ for all $N \in \mathbb{N}$. As $R^D = \bigcup_N B_N$, $\mathcal{S}$ is a SAGBI-basis for $R^D$.
\begin{proof}
$\text{L}_{alg}(\mathcal{S}_N) = \mathbb{K}[b_i, c_i, d_i, e \, \vert \, 0\leq i \leq N ]$. The relations between the $b_i$, $c_i$ and $d_i$ are generated by
\[ \begin{array}{cc}
    b_nc_m -b_{n'}c_{m'} = 0, & b_nb_m -b_{n'}b_{m'} = 0,  \\
    c_nd_m -c_{n'}d_{m'} = 0, & c_nc_m -c_{n'}c_{m'} = 0, \\
    b_nd_m -b_{n'}d_{m'} = 0, & d_nd_m -d_{n'}d_{m'} = 0, \\
\end{array} \]
where $n,m,n',m' \in \mathbb{N}$ satisfy $n+m = n'+m' \leq N$. We also have the relations involving $e$
\[ d_md_n - e\prod_{i=1}^6b_{m_i} =0, \]
with $n + m = \sum_{i=1}^6 m_i \leq N$. The relations between the $b_i, c_i, d_i$ and $e$ all arise by noting that in any relation the terms must have equal $v,t$ and $u$-degree and so there must be an equal number of $b_i$, $c_i$ and $d_i$ terms on either side in any relation involving just these three families. The relations involving $e$ arise from comparing $x$ and $u$-degree. 

We now show that when substituting in the polynomials $\beta_i, \gamma_i, \delta_i$ and $g$, in the relations above, the leading term of the result lies in $\text{L}_{alg}(\mathcal{S}_N)$. By considering the first two terms of the $\beta_i$ and $\gamma_i$ and noting that $m-m' = n'-n$, we see that
\begin{align*}
    \text{LT}(\beta_n \gamma_m - \beta_{n'}\gamma_{m'}) & =  \text{LT}\big((xv^n - nsv^{n-1})((2x^3t - s^2)v^m - m(-3x^5u + x^2st)v^{m-1})  \\
    & \qquad \quad \, - (xv^{n'} - n'sv^{n'-1})((2x^3t - s^2)v^{m'} - m'(-3x^5u + x^2st)v^{m'-1})\big)  \\
    & = \text{LT}\left((2(n'-n)x^3st - (n'-n)s^3 + (m-m')x^3st - 3(m-m')x^6u)v^{n+m-1}\right) \\
    & = \text{LT}\left((m-m')(3x^6u - 3x^3st+s^3 + s^3)v^{n+m-1}\right).
\end{align*}
So $\text{LT}(\beta_n \gamma_m - \beta_{n'}\gamma_{m'}) = -3(m-m')x^6uv^{n+m-1} = (m'-m)d_{n+m-1} \in \text{L}_{alg}(\mathcal{S}_N)$, in fact we have shown that the coefficient of $v$-degree $n+m-1$ is precisely $\delta_0$.
Next we have
\begin{align*}
    \text{LT}(\beta_n \delta_m - \beta_{n'}\delta_{m'}) & =  \text{LT}\big((xv^n - nsv^{n-1})((3x^6su - 3x^3st + s^3)v^m - m(x^5su + 4x^5t^2-x^2s^2t)v^{m-1})  \\
    & \qquad \quad \, - (xv^{n'} - n'sv^{n'-1})((3x^6u - 3x^3st + s^3)v^{m'} - m'(x^5su + 4x^5t^2-x^2s^2t)v^{m'-1})\big)  \\
    & = \text{LT}\left((4(m-m')x^6t^2 + (3n-3n'+ m'-m)x^3s^2t - (n'-n)s^4)v^{n+m-1}\right) \\
    & = \text{LT}\left(((m-m')(4x^6t^2 + x^3s^2t-s^4))v^{n+m-1}\right).
\end{align*}
Therefore $\text{LT}(\beta_n \delta_m - \beta_{n'}\delta_{m'}) = 4(m-m')x^6t^2v^{n+m-1} = (m-m')c_0 c_{n+m-1}$, and the coefficient of $v$-degree $n+m-1$ is precisely $\gamma^2_0$. By the same method, we find that 
\[\text{LT}(\gamma_n \delta_m - \gamma_{n'}\delta_{m'})  = 9(n'-n)x^{11}u^2v^{n+m-1}  = (n'-n)e b_0^4 b_{n+m-1},\]
and the coefficient of $v$-degree $n+m-1$ is precisely $(n'-n)\beta_0^4 g$. Similarly we have 
\[\text{LT}(\delta_n\delta_m - g\prod_{i=1}^6\beta_{m_i})  = -8x^9t^3v^{n+m}  = -c_0^2 c_{n},\] 
and the coefficient of $v$-degree $n$ is precisely $-\gamma_0^3$. Note that this arises from the relation $\gamma_0^3 + \delta_0^2 = x^6g$. Now we require the first three terms of the $\beta_i$, $\gamma_i$ and $\delta_i$ to compute the remaining relations, as the terms of $v$-degree $m+n$ and $m+n-1$ are both zero. We find: 
\begin{align*}
\text{LT}(\beta_n \beta_m - \beta_{n'}\beta_{m'}) & =  \text{LT}\Big(\big(xv^2 - nsv + n(n-1)x^2t)(xv^2 - msv + m(m-1)x^2t)  \\
    & \qquad \quad \,- (xv^2 - n'sv + n'(n'-1)x^2tv)(xv^2 - m'sv+m'(m'-1)x^2t\big)v^{n+m-4}\Big)  \\
    & = \text{LT}\left(((nm-n'm')s^2 + (n(n-1)+m(m-1))x^3s^2t)v^{n+m-2}\right) \\
    & = \text{LT}\left(-(nm-n'm')(2x^3t -s^2)v^{n+m-2}\right), \stepcounter{equation}\tag{\theequation}\label{beta SAGBI}
\end{align*}
using that $n^2+m^2-n'^2-m'^2 = -2(nm-n'm')$, thus $\text{LT}(\beta_n \beta_m - \beta_{n'}\beta_{m'}) = -(nm-n'm')2x^3tv^{n+m-2} = -(nm-n'm')c_{n+m-2}$ and the coefficient of $v$-degree $n+m-2$ is precisely $-(nm-n'm')\gamma_0$.
\begin{align*}
\text{LT}(\gamma_n \gamma_m - \gamma_{n'}\gamma_{m'}) & =  \text{LT}\Big(\big(nm(x^2st-3x^5u)^2 - (n^2-n+m^2-m))(2x^3t-s^2)(3x^4su-2x^4t^2) \\
& \qquad  -n'm'(x^2st-3x^5u)^2 + (n'^2-n'+m'^2-m')(2x^3t-s^2)(3x^4su-2x^4t^2) \big)v^{n+m-2} \Big) \\
& = \text{LT}\Big(\big((nm-n'm')(9x^{10}u^2 - 6x^7stu + x^4s^2t^2) \\
& \qquad + (n^2+m^2-n'^{2}-m'^{2})(6x^7stu-3x^4s^3u-4x^7t^3+2x^4s^2t^2)\big)v^{n+m-2}\Big) \\
& = \text{LT}\left((nm-n'm')x^4(9x^6u^2-18x^3stu+6s^3u+8x^3t^3-3s^2t^2))v^{n+m-2}\right). \stepcounter{equation}\tag{\theequation}\label{gamma SAGBI}
\end{align*}
So $\text{LT}(\gamma_n \gamma_m - \gamma_{n'}\gamma_{m'}) = (nm-n'm')9x^{10}u^2v^{n+m-2} = (nm-n'm')eb_0^3b_{n+m-2}$ and we have shown that the coefficient of $v$-degree $n+m-2$ of this expression is precisely $(nm-n'm')\beta_0^4g$. Now finally we have
\begin{equation} \label{Delta SAGBI}
    \text{LT}(\delta_n \delta_m - \delta_{n'}\delta_{m'}) = \text{LT}((nm-n'm')x^4(2x^3t-s^2)(9x^6u^2-18x^3stu+6s^3u+8x^3t^3-3s^2t^2))v^{n+m-2}).
\end{equation}
We conclude $\text{LT}(\delta_n \delta_m - \delta_{n'}\delta_{m'}) = (nm-n'm')9x^13tu^2v^{n+m-2} = (nm-n'm')b_0^3c_0eb_{n+m-2}$ and we have shown that the coefficient of $v$-degree $n+m-2$ of this expression is precisely $(nm-n'm')\beta_0^4\gamma_0g$. 

Since $\mathcal{S}_N$ generates $B_N$, and any combination of elements in $\mathcal{S}_N$ yields an element whose leading term lies in $\text{L}_{alg}(\mathcal{S}_N)$, we conclude that $\mathcal{S}_N$ is a SAGBI-basis for $B_N$.
\end{proof}
\end{lem}

\section{The finite generation ideal}

\input{fgi_proof}

\noindent {\large \textbf{Acknowledgements}}

\smallskip
\noindent This research was supported by EPSRC funding. I would like to thank both Dr Emilie Dufresne and Professor Michael Bate for their help supervising this work as well as the anonymous referees of the prior version of this paper whose feedback greatly helped to improve the exposition.

\printbibliography

\end{document}

%% file: Prop1.tex
We induct on the degree of $v$. Note that $\beta_0, \gamma_0$ and $\delta_0$ have already been defined. For the sake of brevity we use $\eta_i$ to denote either $\beta_i, \gamma_i$ or $\delta_i$ whenever it is unnecessary to differentiate between them. Now suppose that for all $i \leq n$, we have defined
\begin{align*}
\eta_i = & \, e_i^{(i)} v^i + e_{i-1}^{(i)} v^{i-1} + \cdots + e_0^{(i)} = e_0^{(0)} v^i + \binom{i}{i-1}e_0^{(1)} v^{i-1} + \cdots + e_0^{(i)},
\end{align*}
where $e_j^{(i)} = \binom{i}{j}e_0^{(i-j)}$ for $e \in \{b,c,d\} $ and $0 \leq j \leq i \leq n$. Note also that $D(e_0^{(i)}) = \Delta(e_0^{(i)}) = -x^2e_1^{(i)} = -x^2\binom{i}{1}e_0^{(i-1)}$, since $D(\eta_i) = 0$. Now we define 
\[ f^{\eta}_{n+1} := (-1)^n \left( \eta_0 v^{n+1} - \binom{n+1}{1}\eta_1 v^{n} + \dots + (-1)^n \binom{n+1}{n}\eta_n v \right).\]
We calculate
\begin{align*}
     D(f_{n+1}^{\eta}) & = (-1)^n\left(D \left(\eta_0 v^{n+1} - \binom{n+1}{1}\eta_1 v^{n} + \dots + (-1)^n \binom{n+1}{n}\eta_n v\right) \right) \\
     & = (-1)^nD(\eta_0)v^{n+1}  + (-1)^{n+1}\binom{n+1}{1}D(\eta_1)v^n  + \dots + (-1)^{2n}\binom{n+1}{n}D(\eta_n)v  \\
     & \quad \, + (-1)^n(n+1)D(v)\eta_0 v^n + (-1)^{n+1}n\binom{n+1}{1}D(v)\eta_1 v^{n-1} + \dots + (-1)^{2n}\binom{n+1}{n}D(v)\eta_n \\ 
     & = -(n+1)D(v)\left( (-1)^{n-1}\left(\eta_0 v^n - \binom{n}{1}\eta_1 v^{n-1} + \dots + (-1)^{n-1}\binom{n}{n-1}\eta_{n-1} v \right) - \eta_n \right)  \\
     & = (n+1)D(v)(\eta_n -f^{\eta}_n).
\end{align*}
Where we have used that $D(\eta_i) = 0 $ for all $i \leq n$ and that $i \binom{n+1}{i} = (n+1) \binom{n}{i-1}$.

We now show that $(n+1)D(v)(\eta_n-f^{\eta}_n)$ has $v$ degree $0$. To do so we calculate the coefficient of $v$-degree $n-j$ in this expression above for $0 \leq j<n$. 
Note that the first $j$ terms  $(-1)^{n-1}(\eta_0 v^n + \dots + (-1)^{j-1}\binom{n}{j-1}\eta_{j-1}v^{n-j+1})$ appearing in $f^{\eta}_n$ all have higher $v$-degree and so we may disregard them in our calculations. Examining the term $(-1)^{n+k-1} \binom{n}{k}\eta_k v^{n-k}$ for $k\geq j$ we find 
\[ (-1)^{n+k-1} \binom{n}{k}\eta_k v^{n-k} = (-1)^{n+k-1}\binom{n}{k}(e_k^{(k)}v^k + \dots + e_0^{(k)})v^{n-k}, \]
so the coefficient of $v$-degree $n-j$ for $(-1)^{n+k-1} \binom{n}{k}\eta_kv^{n-k}$ is
\[ (-1)^{n+k-1} \binom{n}{k}e_{k-j}^{(k)} = (-1)^{n+k-1} \binom{n}{k}\binom{k}{k-j} e_0^{(j)} = (-1)^{n+k-1} \binom{n}{j}\binom{n-j}{k-j} e_0^{(j)}. \]
Summing these coefficients, we obtain
\[ \sum_{k=j}^n (-1)^{n+k-1} \binom{n}{j}\binom{n-j}{k-j} e_0^{(j)}  = (-1)^{n+j-1}e_0^{(j)}\binom{n}{j} \left( \sum_{t=0}^{n-j} (-1)^{t} \binom{n-j}{t} \right) = 0, \]
since 
\[ \sum_{t=0}^n (-1)^t \binom{n}{t} = 0 \quad \text{ for }n \geq 1. \]
It remains to calculate the term of $v$-degree $0$  in $(n+1)D(v)(\eta_n -f^{\eta}_n)$, but this is simply $(n+1)x^2 e_0^{(n)}$, and so we have shown
\[ D(f_{n+1}^{\eta}) = (n+1)D(v)(\eta_n - f^{\eta}_n) = (n+1)x^2e_0^{(n)}. \]
Therefore, if we can show $x^2 e_0^{(n)} \in \, \Delta(S)$ then we may define $e_0^{(n+1)}:= -(n+1)h$, where $D(h) = \Delta(h) = x^2 e_0^{(n)}$ and set $\eta_{n+1}:= f^{\eta}_{n+1} + e_0^{(n+1)}$. To achieve this we use a method similar to van den Essen in \cite{VAN} by considering the construction of the element $b'$ in the image membership algorithm. In the notation of Lemma \ref{Image membership algorithm}, we choose our local slice to be $p=s \in S$ so that $d=\Delta(s) = x^3$ and then
\[ b' = \sum_{i=0}^{n} \frac{(-1)^i}{(i+1)!} \Delta^i \left(x^2 e_0^{(n)} \right)\left(\frac{s}{x^3}\right)^{i+1} = \sum_{i=0}^n \frac{1}{i+1} \binom{n}{i} e_0^{(n-i)} s^{i+1} x^{-i-1}. \]
Note that $x^{n+1} b' \in S$ and $\Delta(x^{n+1} b') = x^{n+3} e_0^{(n)}$. Our aim is to show that there is some $h \in S^{\Delta}$ such that $x^{n+1} b' - h = x^{n+1} b$, giving us that $b \in S$ satisfies $\Delta(b) = x^2 e_0^{n}$. To achieve this, we must now treat the $b_0^{(n)}$, $c_0^{(n)}$ and $d_{0}^{(n)}$ separately due to their differing degrees and $\rho$-degrees, though the arguments are very similar. We show the case for $c_0^{(n)}$. In this case $f^{\gamma}_{n+1}$ is homogeneous of degree $2n + 8$, and hence $x^{n+1} b'$ is homogeneous of degree $3n+9$. We now define $S_{(a,k)}$ to be the $\mathbb{K}$-vector space spanned by monomials of degree $a$ and $\rho$-degree $k$, that is 
\[ S_{(a,k)}:= \left \lbrace \sum_i \lambda_i m_i \in S \, \vert \, \lambda_i \in \mathbb{K}, \, \deg(m_i)=a, \, \rho(m_i) = k \right \rbrace. \]
Note that $x^{n+1}b' \in S_{(3n+9,n+1)}$ and
\[\beta_0 = x \in S_{(1,0)}, \quad s \in S_{(3,1)}, \quad t \in S_{(3,2)}, \quad u \in S_{(3,3)}, \quad \gamma_0 \in S_{(6,2)}, \quad \delta_0 \in S_{(9,3)}, \quad g \in S_{(12,6)}.\]
Additionally, note that if $f \in S_{(a,k)}, \,\, g\in S_{(b,l)}$, then $fg \in S_{(a+b,k+l)}$.
Define
\begin{align*} M:=  & \, \, \lbrace f \in S_{(3n+9, n+1)} \,  \vert \, \text{deg}_x D(f) \geq n+1 \rbrace, \\
N:=  &  \, \, S^{\Delta} \cap S_{(3n+9,n+1)},
\end{align*}
so $x^{n+1}b' \in M$. Finally, we define $\pi: S \rightarrow S$, where $\pi(f)$ removes all terms of $f$ of $x$-degree greater than or equal to $n+1$. We show that $\pi(N) = \pi(M)$.

Let $f \in N$, note that $x^{2n+8}s^{n+1} \in M$ and $\pi(x^{2n+8}s^{n+1}) = 0$. If we set $q := x^{2n+8}s^{n+1} + f,$ then $D(q) = D(x^{2n+8}s^{n+1})$, so $q \in M$, and $\pi(q) = \pi(x^{2n+8}s^{n+1}+f) = \pi(f)$, giving us that $\pi(N) \subset \pi(M)$. 

For $h \in M$, write 
\[h = \sum \alpha_{b,c,d} x^a s^b t^c u^d, \quad \text{ where }a+ 3(b+c+d) = 3n+9, \quad b + 2c + 3d = n+1, \] 
then we find
\[ D(h) = \sum (b\alpha_{b,c,d} + (c+1)\alpha_{b-2,c+1,d} + (d+1)\alpha_{b-1, c-1, d+1}) x^{a+3} s^{b-1}t^c u^d. \] 
The condition that deg$_x D(h) \geq n+1$ gives us that 
\[ (b\alpha_{b,c,d} + (c+1)\alpha_{b-2,c+1,d} + (d+1)\alpha_{b-1, c-1, d+1}) =0, \] whenever $a+3 < n+1$. 
We note that whenever $a$ satisfies this inequality we have that $b >0$; thus each $\alpha_{p,q,r}$ is a combination of $\alpha_{b,c,d}$ with $b+c+d < p+q+r$. Therefore each $\alpha_{p,q,r}$ is a linear combination of $\alpha_{b,c,d}$ with $3(b+c+d) = 6l+9$ when $n=3l$, $6l+12$ when $n=3l+1$ and $6l+15$ when $n=3l+2$, giving us cases depending on $n$ mod $3$ using that $a = 3n+9 - 3(b+c+d)$. When counting the number of solutions to this equation and $b+2c+3d = n+1$ these split again to cases for $n$ mod $6$, but we find that in all cases there are exactly $k+1$ solutions, where $n=6k+i$, $0 \leq i \leq 5$. Hence dim$\,(\pi(M)) \leq k+1$. 

Now for $N$, we know that $S^{\Delta} = \mathbb{K}[\beta_0,\gamma_0,\delta_0,g]$, so $N$ is generated by $\beta_0^a \gamma_0^b \delta_0^c g^d$ where $a + 6b + 9c + 12d = 3n+9$, as elements of $N$ are homogeneous of degree $3n+9$ and $\rho$-degree $2b+ 3c + 6d = n+1$. Counting the number of solutions to these equations we find again that these split mod $6$, with $\frac{1}{2}(k+1)(k+2)$ solutions for $n = 6k+i$, $i \in \lbrace 0, 1, 2, 4 \rbrace $ and $\frac{1}{2}(k+2)(k+3)$ solutions for $n=6k+3, 6k+5$. 

Consider the case $n=6k$, the general solution for these equations is $(a,b,c,d) = (6y, 3k-3z,1+2z-2y, y)$ with $0 \leq y \leq z \leq k$ integers. But, since $\gamma_0^3+\delta_0^2 = x^6g$, we obtain, for example, that $\gamma_0^{3k}\delta_0 + \gamma_0^{3k-3}\delta_0^3 = x^6\gamma_0^{3k-3}\delta_0g$. Using this relation we find that all solutions with $y\neq 0$ can be written as a linear combination of solutions with $y=0$, thus reducing the number of solutions to $k+1$. The same argument for other values of $n$ mod $6$ reduces the number of solutions to $k+1$ for $n= 6k+i$, $i \in \lbrace 0,1,2,4 \rbrace$ and $k+2$ for $n= 6k+3, 6k+5$. Returning to $n=6k$, we show that $\pi(N)$ has dimension $k+1$. Let
\[ n_p:= (2x^3t-s^2)^{3k-3p}(3x^6u - 3x^3st +s^3)^{1+2p}, \quad 0 \leq p \leq k, \]
and consider 
\begin{align*} (-1)^p \pi(n_p)\vert_{t=0, u=s/3} & = \pi \left(s^{6k-6p}(x^6 s + s^3)^{1+2p}\right) \\
& = \pi \left( \sum_{i=0}^{2p} \binom{2p}{i} x^{6i} s^{6k+3 - 2i} \right). 
\end{align*}
Now since $n+1 = 6k+1$, $\pi$ removes all monomial terms of $x$-degree $\geq 6k+1$, we therefore find
\[ (-1)^p \pi(n_p)\vert_{t=0, u=s/3} = \sum_{i=0}^{min(2p,k)} \binom{2p}{i} x^{6i} s^{6k+3 - 2i} . \]
So the linear independence of the $\pi(n_p)$ follows once we show that 
\begin{equation*} \label{binom determinant} \text{det}\left(\binom{2p}{i} \right)_{0 \leq i, p \leq k} \neq 0.
\end{equation*}
But this is a special case of Corollary $2$ of \citeauthor{GES}'s article \cite[p.301]{GES}, with $a_i = 2i$ and $b_i = i$ for $i=0,1,\dots,k$.
The cases for $n = 6k +1, 6k+2, 6k+4$ follow similarly. In the case of $n = 6k+5$, we let
\[ n_p:= (2x^3t-s^2)^{3k+4-3p}(3x^6u -3x^3st+s^3)^{2p}, \quad 0 \leq p \leq k+1,\]
again we consider 
\begin{align*} (-1)^p \pi(n_p)\vert_{t=0, u=s/3} & = \pi\left( s^{6k+8-6p}(x^6 s + s^3)^{2p} \right) \\
& = \pi \left( \sum_{i=0}^{2p} \binom{2p}{i} x^{6i}s^{6k+8-2i} \right). 
\end{align*}
Now $n+1 = 6k+6$, and hence $\pi$ removes all monomial terms of $x$-degree $\geq 6k+6$, so
\[ (-1)^p \pi(n_p)\vert_{t=0, u=s/3} = \sum_{i=0}^{min(2p,k)} \binom{2p}{i} x^{6i} s^{6k+8 - 2i}. \]
Considering only the first $k+1$ terms, as above, we find these $\pi(n_p)$ are linearly independent by \citeauthor{GES}'s result. Using that $\pi(N) \subset \pi(M)$, and dim$(\pi(M)) \leq k+1$, we conclude that $\pi(N) = \pi(M)$.

%% file: fgi_proof.tex
We maintain our notation for $B_N$ and $\mathcal{S}_N$ introduced in the previous section. Our aim in this section will be to prove the following: 
\begin{thm} \label{finite generation ideal}
The finite generation ideal, $\mathfrak{f}_{R^D}$, is the radical of the ideal of $R^D$ generated by $\beta_0, \gamma_0$ and $\delta_0$; that is, $\mathfrak{f}_{R^D} = \sqrt{(\beta_0, \gamma_0,\delta_0)R^D}$. Additionally, $\mathcal{G} = \lbrace \beta_i, \gamma_i, \delta_i \, \vert \, i \geq 0 \rbrace \subset \mathfrak{f}_{R^D}$ satisfies $\text{L}(\mathcal{G}) \subset \text{L}(\mathfrak{f}_{R^D})$.
\end{thm}
To prove the theorem, we first prove the following lemma and proposition, analogues of results proven in \citeauthor{DUF}'s paper \cite[p.21]{DUF} in order to compute the finite generation ideal of Roberts' example. Recall that for subalgebras $S_1 \subset S_2 \subset R$, the \emph{conductor} is defined by $[S_1 \colon S_2] := \lbrace s \in S_2 \, \vert \, sS_2 \subset S_1 \rbrace$. 
\begin{lem} \label{conductor} \begin{enumerate}
    \item If $f\in R^D$ and \emph{deg}$_v (f) \leq N$, then $f \in B_N$.
    \item 
    $(\beta_0,\gamma_0,\delta_0)B_{N+1} \subset B_N$.
    \item $\left[B_N \colon B_{N+1} \right] \cap B_0 = (\beta_0, \gamma_0, \delta_0) B_0$.
    \end{enumerate}
    \begin{proof}
If $\text{deg}_v(f) = 0,$ then $D(f)= \Delta(f) = 0$ so $f \in S^{\Delta}$ which is generated by $\beta_0, \gamma_0, \delta_0$ and $g$. But $B_0 = \mathbb{K}[\mathcal{S}_0]$ where $\mathcal{S}_0 = \lbrace \beta_0, \gamma_0, \delta_0, g \rbrace$, so $f \in B_0$. Suppose that this result holds for all $f \in R^D$ with $\deg_v(f) \leq k$. Now suppose that $f \in R^D$ and  $\deg_v(f)= k+1$, then LT$(f)$ is a monomial in $\text{L}_{alg}(S)$ of $v$-degree $k+1$, and hence there is some $\Tilde{f} \in B_{k+1}$ with $\text{LT}(f)=\text{LT}(\Tilde{f})$. 
But deg$_v(f- \Tilde{f}) < k+1$, so $f-\Tilde{f} \in B_k \subset B_{k+1}$ by induction. Hence
\[ f = \Tilde{f} + (f - \Tilde{f}) \in B_{k+1}. \] This proves part $1$.

For part $2$, let $\eta_0 \in \lbrace \beta_0, \gamma_0, \delta_0 \rbrace$ and consider $\eta_0 B_{N+1}$. Since $g, \beta_i, \gamma_i, \delta_i \in B_N$ for $0 \leq i \leq N$ we need only show that $\eta_0\beta_{N+1}, \eta_0\gamma_{N+1}, \eta_0\delta_{N+1} \in B_N$. Let LT$(\eta_0) = e_0$. Now
\[ \begin{array}{ll}
\text{LT}(\eta_0\beta_{N+1}) = e_0b_{N+1} = e_1b_N, & \text{deg}_v(\eta_0\beta_{N+1} - \eta_1 \beta_N) \leq N,  \\
\text{LT}(\eta_0\gamma_{N+1}) = e_0c_{N+1} = e_1c_N, & \text{deg}_v(\eta_0\gamma_{N+1} - \eta_1 \gamma_N) \leq N, \\ 
\text{LT}(\eta_0\delta_{N+1}) = e_0d_{N+1} = e_1d_N, & \text{deg}_v(\eta_0\delta_{N+1} - \eta_1 \delta_N) \leq N. \\
\end{array} \]
Applying part $1$ of this lemma in each case gives us that:
\[ \begin{array}{l}
\eta_0\beta_{N+1} = \eta_1 \beta_N + (\eta_0\beta_{N+1} - \eta_1 \beta_N) \in B_N, \\
\eta_0\gamma_{N+1} = \eta_1 \gamma_N + (\eta_0\gamma_{N+1} - \eta_1 \gamma_N) \in B_N, \\
\eta_0\delta_{N+1} = \eta_1 \delta_N + (\eta_0\delta_{N+1} - \eta_1 \delta_N) \in B_N.
\end{array} \] This proves part $2$.

Finally for part $3$, note that if $f \in [B_N \colon B_{N+1}] \cap B_0$, then $f\beta_{N+1}, f\gamma_{N+1}, f\delta_{N+1} \in B_{N}$. Therefore all three of LT$(f)xv^{N+1},$ LT$(f)2x^3tv^{N+1}$ and LT$(f)3x^6u $ are elements of $\text{L}_{alg}(B_N)$ 
which must each have at least two factors of the form $b_i, c_i, d_i$ for some $0 \leq i \leq N$. Now LT$(f)$, as a monomial in $L_{alg}(S_0)$ must therefore contain a factor $b_0, c_0$ or $d_0$, call this $e_0$. Then LT$(f) = e_0$LT$(\Tilde{f})$ for some $\Tilde{f} \in B_0$, giving $f-e_0\Tilde{f} < f$ and our result follows by induction since we have shown $\beta_0,\gamma_0,\delta_0 \in [B_N,B_{N+1}]$ in part $2$.
\end{proof}
\end{lem}

\begin{prop} \label{radical proof} Let $f \in R^D$ be a homogeneous invariant with $f \neq g^k$ for any $k \in \mathbb{N}$. Then $f \in \sqrt{(\beta_0,\gamma_0,\delta_0)R^D}$, and hence $\beta_i, \gamma_i, \delta_i \in \sqrt{(\beta_0,\gamma_0,\delta_0)R^D}$ for all $i \in \mathbb{N}$. Furthermore, $\sqrt{(\beta_0,\gamma_0,\delta_0)R^D}$ is generated by $\lbrace \beta_i, \gamma_i, \delta_i \rbrace_{i \in \mathbb{N}}$, with $R^D / \sqrt{(\beta_0,\gamma_0,\delta_0)R^D}$ a polynomial ring in one variable.
\end{prop}
To prove this result, we begin by first showing that $g \notin \sqrt{(\beta_0, \gamma_0, \delta_0)R^D}$. Suppose that 
\[ g^k = \beta_0 p_1 + \gamma_0 p_2 + \delta_0 p_3, \]
for some $k \in \mathbb{N}$. We have $\deg(g^k) = 12k$ and $\rho(g^k) = 6k$. Now we may suppose that $p_1 \in R^D$ is homogeneous, with degree $12k-1$ and $\rho$-degree $6k$; but there is simply no invariant in $R^D$ which has both this corresponding degree and $\rho$-degree. This same argument holds for the degrees and $\rho$-degrees of both $p_2$ and $p_3$. Therefore $g \notin \sqrt{(\beta_0, \gamma_0, \delta_0)R^D}$ as claimed. 

We focus now on showing that $\beta_i, \gamma_i$ and $\delta_i \in \sqrt{(\beta_0,\gamma_0,\delta_0)R^D}$. Since doing so for $\beta_i$ requires showing that $\gamma_i \in \sqrt{(\beta_0,\gamma_0,\delta_0)R^D}$ we start with $\gamma_i$. We use the equation \ref{gamma SAGBI} from Lemma \ref{SAGBI} to examine the expression $\gamma_i\gamma_j - \gamma_{i'}\gamma_{j'}$. The $v$-degree $0$ part of this expression is $c_0^{(i)}c_0^{(j)} - c_0^{(i')}c_0^{(j')}$. Our goal will be to construct an invariant, $\xi$, which has the same $v$-degree $0$ terms. In doing so we observe that $\gamma_i\gamma_j - \gamma_{i'}\gamma_{j'} - \xi$ is an invariant with no $v$-degree $0$ terms, therefore $\gamma_i\gamma_j - \gamma_{i'}\gamma_{j'} - \xi = v\mu$ and $D(v\mu) = x^2\mu + vD(\mu) = 0 $. Comparing the $v$-degree $0$ terms of this expression, we find that $\mu$ has no terms of $v$-degree $0$, and hence $\gamma_i\gamma_j - \gamma_{i'}\gamma_{j'} - \xi$ has no terms of $v$-degree $1$. Continuing in this way, we see that we must conclude $\gamma_i\gamma_j - \gamma_{i'}\gamma_{j'} = \xi$. If we can show that $\xi \in (\beta_0, \gamma_0, \delta_0)R^D$ for all $i,j$, then we have in particular $(\gamma_i)^2 = \gamma_0\gamma_{2i} + \xi \in (\beta_0, \gamma_0, \delta_0)R^D$, giving $\gamma_i \in \sqrt{(\beta_0, \gamma_0, \delta_0)R^D}$ as required. In the following, we let $C$ be the vector space whose basis is given by finite combinations of $b_0^{(i)},c_0^{(i)}, d_0^{(i)}$ and $g$. Recall $b_0^{(i)},c_0^{(i)}, d_0^{(i)}$ are the $v$-degree $0$ components of $\beta_i, \gamma_i$ and $\delta_i$ respectively.

\begin{lem} \label{gamma proof}
 Fix $n \in \mathbb{N}$, with $n \geq 2$, then $c_0^{(i)}c_0^{(j)} - c_0^{(i')}c_0^{(j')} = -\lambda_{i,i'}^n x^3 g b_0^{(n-2)} + r(n,i,i')$, where $r(n,i,i') = b_0^{(0)}h_1 + c_0^{(0)}h_2 + d_0^{(0)}h_3$, $h_i \in C$ for $i = 1,2,3$, and $\lambda_{i,i'}^n := -(ij - i'j')$, $j=n-i$, $j' = n-i'$.
 \end{lem}
 \begin{proof}
 We prove this result by induction. First, for $n=2$ we have
 \begin{align*} c_0^{(0)}c_0^{(2)} - \left(c_0^{(1)}\right)^2 & = 12x^7stu -6x^5s^3u -8x^7t^3 +4x^4s^2t^2 -9x^{10}u^2 + 6x^7stu-x^4s^2t^2 \\
 & = -9x^{10}u^2 + 18x^7stu -6x^4s^3u -8x^7t^3 + 3x^4s^2t^2 \\ 
 & = x^4g. 
 \end{align*}
We also note that, as can be seen from equation \ref{gamma SAGBI}, the coefficient of $v$-degree $n-2$ in the expression $\gamma_i \gamma_j - \gamma_{i'} \gamma_{j'}$ is precisely $\lambda_{i,i'}^n\beta_0^4 g$.

Now, suppose that the result holds for all pairs $i,j$ with $i + j = n$, we consider 
\arraycolsep=0.5pt \def\arraystretch{2.2} \begin{equation*}
\begin{split}
D\Big(c_0^{(i+1)}&c_0^{(j)} - c_0^{(i'+1)}c_0^{(j')} \Big) \\
 = -x^2 \Big( & (i+1)c_0^{(i)}c_0^{(j)} + jc_0^{(i+1)}c_0^{(j-1)} - (i'+1)c_0^{(i')}c_0^{(j')} - j'c_0^{(i'+1)}c_0^{(j'-1)} \Big) \\
 = -x^2\Big(&(i+1) \left(c_0^{(i)}c_0^{(j)} - c_0^{(i')}c_0^{(j')} \right) -(i'-i)c_0^{(i')}c_0^{(j')} \\
 & \, + j\left(c_0^{(i+1)}c_0^{(j-1)} - c_0^{(i'+1)}c_0^{(j'-1)}\right) - (j'-j)c_0^{(i'+1)}c_0^{(j'-1)} \Big) \\
 = -x^2 \Big(& (i+1) \left( c_0^{(i)}c_0^{(j)} - c_0^{(i')}c_0^{(j')}\right) + j\left(c_0^{(i+1)}c_0^{(j-1)} - c_0^{(i'+1)}c_0^{(j'-1)}\right) \\ 
 & \, -(i'-i)\left(c_0^{(i')}c_0^{(j')} - c_0^{(i'+1)}c_0^{(j'-1)}\right) \Big) \\
 = -x^2 \Big(& (i+1)\left(\lambda_{i,i'}^n x^3gb_0^{(k-2)}+r(n,i,i') \right) + j \left(\lambda_{i+1,i'+1}^n x^3gb_0^{(n-2)} + r(n,i+1,i'+1) \right) \\
 & \, -(i'-i) \left(\lambda_{i',i'+1}^n x^3gb_0^{(n-2)} + r(n,i',i'+1) \right) \Big) \\
 = -x^2 \Big(& x^3gb_0^{(n-2)} \left( (i+1)\lambda_{i,i'}^n + j\lambda_{i+1,i'+1}^n - (i'-i)\lambda_{i',i'+1}^n \right) + (i+1) r(n,i,i') \\ 
 & \, + jr(n,i+1,i'+1) -(i'-i) r(n,i',i'+1) \Big) \\
  = -x^2 \Big(& (n-1)\lambda_{i+1,i'+1}^{n+1}x^3gb_0^{(n-2)} + (i+1) r(n,i,i') + j r(n,i+1,i'+1) \\
 & \, -(i'-i) r(n,i',i'+1) \Big). \end{split}
\end{equation*}
If either $i=n$ or $i'=n$ we instead obtain either $(i'+1)r(n,n,i') +j'r(n,n,i'+1)$ or $(i+1)r(n,i,n) + jr(n,i+1,n)$ in place of the other $r(n,a,b)$. For all $n \geq 2$, and all $a,b$ we claim that there is some $R(n+1,a,b) = b_0^{(0)}H_1 + c_0^{(0)}H_2 + d_0^{(0)}H_3$, $H_i \in C$ for $i=1,2,3$ with $D(R(n+1,a,b))= -x^2r(n,a,b)$. Let $C_N \subset C$ be the vector space whose basis is given by finite combinations $e_0^{(a_1)}\cdots e_0^{(a_k)}$, where $\sum_{i=1}^k a_i = N$ and each $e$ appearing is one of $b, c$ or $d$, not necessarily all the same. Consider a term of $r(n,a,b)$, by which we mean an element of the form $\lambda e_0^{(0)}h$, with  $\lambda \in \mathbb{K}, \,\, h = e_0^{(a_1)}\cdots e_0^{(a_k)}g^l \in C_{N}$. Note that the expression $c_0^{(i+1)}c_0^{(j)} - c_0^{(i'+1)}c_0^{(j')}$ is homogeneous of degree $2n+14$ and $\rho$-degree $n+5$, so $h \neq 0$ for $n \geq 2$. Additionally $h \neq g^l$ since this would give the degree of $r(n,a,b)$ as $\deg(e_0^{(0)}) + 12l $, and $\rho$-degree $\rho (e_0^{(0)}) + 6l$ which cannot be $2n+14$ and $n+5$ respectively for any choice of $e_0^{(0)}$.  
We write $h = g^l h'$, where $h' \in C_N$ for some $N$. Since $\Delta(e_0^{(N)}) = -x^2 N e_0^{(N-1)}$ we can consider $\Delta$ as a linear map
\[ \Delta: C_{N+1} \longrightarrow x^2C_N. \]
We now show that $\Delta$ is surjective, in which case we can find $H \in C_{N+1}$ such that $D(H) = -x^2h$, and $D(\lambda e_0^{(0)}H) = -\lambda x^2e_0^{(0)}h$. Repeating this process for all terms of $r(n,a,b)$ then gives us $R(n+1,a,b)$. To show $\Delta$ is surjective it is sufficient to show that for all $f = x^2e_0^{(a_1)}\cdots e_0^{(a_k)} \in x^2C_{N}$, we have $f \in \Delta(C_{N+1})$. We describe the process of constructing an element $F$ with $\Delta(F) = f$. 

Firstly we let 
\[F_1:= -\frac{1}{a_1+1}e_0^{(a_1+1)}\cdots e_0^{(a_k)}, \] then $\Delta(F_1) = f + G_1$, where all terms of $G_1$ are of the form $-x^2 \lambda e_0^{(a_1+1)}\cdot e_0^{(b_2)}\cdots e_0^{(b_k)}$ with $\lambda \in \mathbb{K}$, $\sum_{i=2}^k b_i = N - a_1 - 1$. Now set 
\[F_2:=F_1 + \sum \kappa e_0^{(a_1+2)}\cdot e_0^{(b_2)}\cdots e_0^{(b_k)}. \] 
Note that $\Delta \left(e_0^{(a_1+2)}\cdot e_0^{(b_2)}\cdots e_0^{(b_k)} \right)$ contains precisely one term of the form $e_0^{(a_1+1)}\cdot e_0^{(b_2)}\cdots e_0^{(b_k)}$. The remaining terms are of the form $e_0^{(a_1+2)}\cdot e_0^{(c_2)}\cdots e_0^{(c_k)}$, where $\sum_{i=2}^k c_i = N - a_1 -2$. Since this is the case we can choose $\kappa$ appearing in $F_2$ so that $\Delta(F_2)$ contains no terms of the form $e_0^{(a_1+1)}\cdot e_0^{(b_2)}\cdots e_0^{(b_k)}$. Continuing in this way we find
\[ D(F_{N-a_1-1}) = f + \omega e_0^{(a_1 + \dots + a_k -1)}\cdot e_0^{(0)}\cdots e_0^{(0)}, \] 
where $\omega \in \mathbb{K}$. Finally we define $F_{N-a_1}:= F_{N-a_1-1} + \frac{\omega}{N} e_0^{(a_1 + \dots + a_k)}\cdot e_0^{(0)}\cdots e_0^{(0)}$ and observe that $D(F_{N-a_1}) = f$ as required.

Having constructed $R(n+1,a,b)$ for all $n, a$ and $b$ we note that
\begin{align*} D \Big( c_0^{(i+1)}c_0^{(j)}  - c_0^{(i'+1)}c_{0}^{(j')} - \lambda_{i+1,i'+1}^{n+1} x^3gb_0^{(n-1)}  - (i+1) R(n+1,i,i') & - j R(n+1, i+1, i'+1) \\
& + (i'-i) R(n+1,i',i'+1) \Big) = 0. 
\end{align*}
Therefore this expression is an invariant of degree $2n+14$ and $\rho$-degree $n+5$. If we consider all homogeneous invariants with such degree and $\rho$-degree we find:
 \[\arraycolsep=4pt\def\arraystretch{1} \begin{array}{l | l}
n & S^{\Delta} \cap S_{(2n+14, n+5)} \\ \hline
 6l    & 0 \\
 6l+1  & \lbrace \lambda x^4g^{l+1} \, \vert \, \lambda \in \mathbb{K} \rbrace \\
 6l+2  & 0 \\
 6l+3  & \lbrace \lambda x^2(2x^3t-s^2)g^{l+1}\, \vert \, \lambda \in \mathbb{K} \rbrace \\
 6l+4  & \lbrace \lambda x(3x^6u-3x^3st+s^3)g^{l+1} \, \vert \, \lambda \in \mathbb{K} \rbrace  \\
 6l+5  & \lbrace \lambda (2x^3t-s^2)^2g^{l+1} \, \vert \, \lambda \in \mathbb{K} \rbrace  \\
\end{array}\] 
Note that all of these elements are of the form $b_0^{(0)}h_1 + c_0^{(0)}h_2 +d_0^{(0)}h_3$, with $h_i \in C$ for all $i$. Thus we can write  
\begin{align*} c_0^{(i+1)}c_0^{(j)}  - c_0^{(i'+1)}c_{0}^{(j')} & - \lambda_{i+1,i'+1}^{n+1} x^3gb_0^{(n-1)} - (i+1) R(n+1,i,i') \\ & - j R(n+1, i+1, i'+1) - (i'-i) R(n+1,i',i'+1) = \mu p, \end{align*}
with $\mu \in \mathbb{K}$ and $p \in S^{\Delta} \cap S_{(2n+14, n+5)}$. By setting
\begin{align*} r(n+1,i,i'): =  & \,\lambda_{i+1,i'+1}^{n+1} x^3gb_0^{(n-1)} - (i+1) R(n+1,i,i') \\ & - j R(n+1, i+1, i'+1) - (i'-i) R(n+1,i',i'+1) - \mu p, 
\end{align*}
then $r(n+1,i,i')$ is of the form $b_0^{(0)}h_1 + c_0^{(0)}h_2 + d_0^{(0)}h_3$, $h_i \in C$ for $i=1,2,3$ and we obtain the required result.
\end{proof}
Now using this proof we consider the following expression
\[ (\gamma_n)^2 - \gamma_{2n}\gamma_{0} - \lambda_{n,2n}^{2n} x^3g \beta_{2n-2} - T(2n, n, 2n), \]
where $T(n,a,b) \in R^D$ is defined by replacing every $e_0^{(k)}$ in $r(n,a,b)$ by the corresponding $\eta_k \in R^D$ which has $e_0^{(k)}$ as its $v$-degree zero term. From this and the observation made above we see that 
\[ (\gamma_n)^2 = \gamma_{2n}\gamma_{0} + \lambda_{n,2n}^{2n} x^3g \beta_{2n-2} + T(2n, n, 2n) \in (\beta_0, \gamma_0, \delta_0)R^D. \]
We now prove a similar result for $\delta_n$.
 
\begin{lem}
  Fix $n \in \mathbb{N}$, with $n \geq 2$, then $ d_0^{(i)}d_0^{(j)} - d_0^{(i')}d_0^{(j')} = \lambda_{i,i'}^n x^3(2x^3t-s^2)g a_0^{(n-2)} + r(n,i,i')$, where $r(n,i,i') = b_0^{(0)}h_1 + c_0^{(0)}h_2 + d_0^{(0)}h_3$, $h_i \in C$ for $i = 1,2,3$, and $\lambda_{i,i'}^n := -(ij - i'j')$, $j=n-i$, $j' = n-i'$.
\end{lem}
 \begin{proof}
 Firstly for $n=2$, we have
\[ d_0^{(0)}d_0^{(2)} - \left( d_0^{(1)} \right)^2 = x^4(2x^3t -s^2)g.\]
Now suppose that the formula holds for all pairs $i+j =n$, we have
\begin{align*}
D \left(d_0^{(i+1)}d_0^{(j)} - d_0^{(i'+1)}d_0^{(j')}\right) = -x^2 \Big(  (n-1)\lambda_{i+1,i'+1}^{n+1} & x^3(2x^3t-s^2)g a_0^{(n-2)} + (i+1) r(n,i,i') \\
& + j r(n,i+1,i'+1) - (i'-i) r(n,i',i'+1) \Big).
\end{align*} 
If either $i=n$ or $i'=n$ we instead obtain either $(i'+1)r(n,n,i') +j'r(n,n,i'+1)$ or $(i+1)r(n,i,n) + jr(n,i+1,n)$ in place of the other $r(n,a,b)$. As before we show that there is some $R(n+1,a,b)$ with $D(R(n+1,a,b)) = r(n,a,b)$ for all $a,b \leq n$. Let $\lambda e_0^{(0)}h$ be a term of $r(n,a,b)$, $\lambda \in \mathbb{K}, \, h \in I$ and  $e \in \lbrace b, c, d \rbrace$. Note that the expression $d_0^{(i+1)}d_0^{(j)} - d_0^{(i'+1)}d_0^{(j')}$ is homogeneous of degree $2n+20$ and $\rho$-degree $n+7$, so $h \neq 0$ for $n \geq 2$. Now $h \neq g^k$ for some $k \in \mathbb{N}$ since $\deg(e_0^{(0)}) + 12k$ and $\rho(e_0^{(0)}) + 6k$ cannot be $2n+20$ and $n+7$ for any choice of $e_0^{(0)}$ or $k$. Therefore we can write $h = f_0^{(l)}h'$ for some $f \in \lbrace b, c, d \rbrace, l \in \mathbb{N}$ and proceed as described in the proof of Lemma \ref{gamma proof}. Now since 
\begin{align*} D \Big( d_0^{(i+1)}d_0^{(j)}  - d_0^{(i'+1)}d_{0}^{(j')} - & \lambda_{i+1,i'+1}^{n+1} x^3(2x^3t-s^2)gb_0^{(n-1)}
  - (i+1) R(n+1,i,i')  \\ & - j R(n+1, i+1, i'+1) + (i'-i) R(n+1,i',i'+1) \Big) = 0, 
\end{align*}
this expression is then an invariant of degree $2n+20$ and $\rho$-degree $n+7$.
Considering all such elements we find: 
 \[\begin{array}{l|l}
n & S^{\Delta} \cap S_{(2n+20, n+7)} \\ \hline
 6l    & 0 \\
 6l+1  & \lbrace \lambda x^4(2x^3t-s^2)g^{l+1} \, \vert \, \lambda \in \mathbb{K} \rbrace \\
 6l+2  & \lbrace \lambda x^3(3x^6u-3x^3st+s^3)g^{l+1} \, \vert \, \lambda \in \mathbb{K} \rbrace \\
 6l+3  & \lbrace \lambda x^2(2x^3t-s^2)^2g^{l+1} \, \vert \, \lambda \in \mathbb{K} \rbrace \\
 6l+4  & \lbrace \lambda x(2x^3t-s^2)(3x^6u-3x^3st+s^3)g^{l+1} \, \vert \, \lambda \in \mathbb{K} \rbrace \\
 6l+5  & \lbrace \lambda (2x^3t-s^2)^3g^{l+1} + \mu (3x^6u-3x^3st+s^3)^2g^{l+1} \, \vert \, \lambda, \mu \in \mathbb{K} \rbrace \\
\end{array}\] 
Note that all such elements are of the form $b_0^{(0)}h_1 + c_0^{(0)}h_2 + d_0^{(0)}h_3$, $h_i \in C$. If we let
\begin{align*} d_0^{(i+1)}d_0^{(j)}  - d_0^{(i'+1)}d_{0}^{(j')} - & \lambda_{i+1,i'+1}^{n+1} x^3(2x^3t-s^2)gb_0^{(n-1)}
  - (i+1) R(n+1,i,i')  \\ & - j R(n+1, i+1, i'+1) + (i'-i) R(n+1,i',i'+1) = \mu p, 
\end{align*}
 with $\mu \in \mathbb{K}$ and $p \in S^{\Delta} \cap S_{(2n+14, n+5)}$, by setting
 \begin{align*} r(n+1,i,i'): =   \lambda_{i+1,i'+1}^{n+1} & x^3(2x^3t-s^2)gb_0^{(n-1)}  - (i+1) R(n+1,i,i') \\ & - j R(n+1, i+1, i'+1) - (i'-i) R(n+1,i',i'+1) - \mu p,  
 \end{align*}
then $r(n+1,i,i')$ is of the form $b_0^{(0)}h_1 + c_0^{(0)}h_2 + d_0^{(0)}h_3$, $h_i \in C$ for $i=1,2,3$, and we obtain the required result.
 \end{proof}
 Now as before by using this proof we consider the following expression
 \[ (\delta_n)^2 - \delta_{2n}\delta_{0} - \lambda_{n,2n}^{2n} x^3(2x^3t-s^2)g \beta_{2n-2} - T(2n, n, 2n), \]
 where $T(n,a,b) \in R^D$ is defined by replacing every $e_0^{(k)}$ in $r(n,a,b)$ by the corresponding $\eta_k \in R^D$ which has $e_0^{(k)}$ as its $v$-degree zero term. From this we see that the expression above has no $v$-degree $0$ terms, and therefore the whole expression must be zero. Thus we have that $(\delta_n)^2 \in \left( \beta_0, \gamma_0, \delta_0 \right)R^D$. Now all that remains is to show that $\beta_n \in \sqrt{(\beta_0, \gamma_0, \delta_0)R^D}$. Firstly, we proceed as we have before for $\gamma_n$ and $\delta_n$:
 
 \begin{lem}
   Fix $n \in \mathbb{N}$, with $n \geq 2$, then $b_0^{(i)}b_0^{(j)}-b_0^{(i')}b_0^{(j')} = \lambda_{i,i'}^kc_0^{(k-2)} + r(n,i,i')$, where \newline $r(n,i,i') = b_0^{(0)}h_1 + c_0^{(0)}h_2 + d_0^{(0)}h_3$, $h_i \in C$ for $i = 1,2,3$, and $\lambda_{i,i'}^n := -(ij - i'j')$, $j=n-i$, $j' = n-i'$.
 \end{lem}
 \begin{proof}
 Firstly note that for $n=2$
\[ \left(b_0^{(1)} \right)^2 - b_0^{(2)}b_0^{(0)} = s^2 - 2x^3t = -c_0^{(0)}.  
\]
Now assuming that the result holds for all pairs $i+j =n$, we compute
 \begin{align*}
D \left(b_0^{(i+1)}b_0^{(j)} - b_0^{(i'+1)}b_0^{(j')}\right) = -x^2 \Big(  (n-1)\lambda_{i+1,i'+1}^{n+1} & c_0^{(n-2)} + (i+1) r(n,i,i') \\
& + j r(n,i+1,i'+1) - (i'-i) r(n,i',i'+1) \Big).
\end{align*} 
If either $i=n$ or $i'=n$ we instead obtain either $(i'+1)r(n,n,i') +j'r(n,n,i'+1)$ or $(i+1)r(n,i,n) + jr(n,i+1,n)$ in place of the other $r(n,a,b)$. As before we show that there is some $R(n+1,a,b)$ with $D(R(n+1,a,b)) = r(n,a,b)$ for all $a, b \leq n$. Let $\lambda e_0^{(0)}h$ be a term of $r(n,a,b)$, $\lambda \in \mathbb{K},\, h \in I$ and  $e \in \lbrace b, c, d \rbrace$. Note that the expression $b_0^{(i+1)}b_0^{(j)} - b_0^{(i'+1)}b_0^{(j')}$ is homogeneous of degree $2n+3$ and $\rho$-degree $n+1$. Therefore $h \neq 0$ for $n \geq 2$, and $h \neq g^k$ for some $k \in \mathbb{N}$ as $\deg(e_0^{(0)}) + 12k$ and $\rho(e_0^{(0)}) + 6k$ cannot be $2n+3$ and $n+1$ for any choice of $e_0^{(0)}$ or $k$. Therefore we can write $h = f_0^{(l)}h'$ for some $f \in \lbrace b, c, d \rbrace$ and $l \in \mathbb{N}$, we can then proceed as described in the proof of Lemma \ref{gamma proof}. Now since
\begin{align*} D \Big( b_0^{(i+1)}b_0^{(j)}  - b_0^{(i'+1)}b_{0}^{(j')} & - \lambda_{i+1,i'+1}^{n+1} c_0^{(n-1)}
  - (i+1)R(n+1,i,i')  \\ & - j R(n+1, i+1, i'+1) + (i'-i) R(n+1,i',i'+1) \Big) = 0, 
\end{align*}
this expression is then an invariant of degree $2n+3$ and $\rho$-degree $n+1$.
Considering all such elements we find: 
\[\begin{array}{l|l}
n & S^{\Delta} \cap S_{(2n+3,n+1)} \\ \hline
 6l    & 0 \\
 6l+1  & \lbrace \lambda (2x^3t-s^2)g^l \, \vert \, \lambda \in \mathbb{K} \rbrace \\
 6l+2  & 0 \\
 6l+3  & 0 \\
 6l+4  & 0 \\
 6l+5  &  \lbrace \lambda x^2g^{l+1} \, \vert \, \lambda \in \mathbb{K} \rbrace \\
\end{array}\]
Note that all such elements are of the form $b_0^{(0)}h_1 + c_0^{(0)}h_2 + d_0^{(0)}h_3$, $h_i \in C$, and so if we let
\begin{align*} b_0^{(i+1)}b_0^{(j)}  - b_0^{(i'+1)}b_{0}^{(j')} & - \lambda_{i+1,i'+1}^{n+1} c_0^{(n-1)}
  - (i+1) R(n+1,i,i')  \\ & - j R(n+1, i+1, i'+1) + (i'-i) R(n+1,i',i'+1) = \mu p, 
\end{align*}
 with $\mu \in \mathbb{K}$ and $p \in S^{\Delta} \cap S_{(2n+14, n+5)}$. By setting
 \begin{align*} r(n+1,i,i'): =   \lambda_{i+1,i'+1}^{n+1} & c_0^{(n-1)} - (i+1) R(n+1,i,i') \\ & - j R(n+1, i+1, i'+1) - (i'-i) R(n+1,i',i'+1) - \mu p,  
 \end{align*}
then $r(n+1,i,i')$ is of the form $b_0^{(0)}h_1 + c_0^{(0)}h_2 + d_0^{(0)}h_3$, $h_i \in C$ for $i=1,2,3$, and we obtain the required result.
 \end{proof}
Now as before by using this proof we consider the following expression
 \[ (\beta_n)^2 - \beta_{2n}\beta_{0} - \lambda_{n,2n}^{2n} \gamma_{2n-2} - T(2n, n, 2n), \]
 where $T(n,a,b) \in R^D$ is defined by replacing every $e_0^{(k)}$ by the corresponding $\eta_k \in R^D$ which has $e_0^{(k)}$ as its $v$-degree zero term. From this we see that the expression above has its $v$-degree $0$ term as $0$, and therefore the whole expression must be zero, as the expression is an invariant, and there is no invariant which is divisible by $v$. This means we have
 \[ (\beta_n)^2 = \beta_{2n}\beta_{0} + \lambda_{n,2n}^{2n} \gamma_{2n-2} + T(2n, n, 2n), \]
 with both $\beta_{2n}\beta_{0}, T(2n,n,2n) \in (\beta_0,\gamma_0,\delta_0)R^D$. We then square our expression to obtain
 \[ (\beta_n)^4 = (\lambda_{n,2n}^{2n})^2 (\gamma_{2n-2})^2 + p, \]
 where $p \in (\beta_0,\gamma_0,\delta_0)R^D$. Using our relations for the $\gamma_i$ calculated in Lemma \ref{gamma proof} we then have
 \[ (\beta_n)^4  = (\lambda_{n,2n}^{2n})^2 \left( (\gamma_{2n-2})^2-\gamma_{2n}\gamma_{0} \right) + p + (\lambda_{n,2n}^{2n})^2\gamma_{2n}\gamma_{0}.\]
Each term on the right-hand side is in $(\beta_0,\gamma_0,\delta_0)R^D$ and hence $(\beta_n)^4 \in (\beta_0,\gamma_0,\delta_0)R^D$. This concludes our proof of Proposition \ref{radical proof} and we are finally able to prove Theorem \ref{finite generation ideal}.  

\begin{proof}[Proof of Theorem \ref{finite generation ideal}]
First we remark that $\beta_0, \gamma_0, \delta_0 \in \mathfrak{f}_{R^D}$ since $\mathfrak{pl}(D) \subset \mathfrak{f}_{R^D}$. Indeed, given $d \in \mathfrak{pl}(D)$, with $d=D(p)$ we have $D(\frac{p}{d})=1$ and the morphism
\[ \frac{p}{d}: \mathbb{A}^5_{d} \longrightarrow \mathbb{G}_a \]
is $\mathbb{G}_a$-equivariant. Hence the affine open set $\mathbb{A}^5_d$ is a trivial $\mathbb{G}_a$-bundle, and $\mathbb{A}^5_d/\mathbb{G}_a = \text{Spec}(\mathbb{K}[x,s,t,u,v]_d^{\mathbb{G}_a})$. Thus $\mathbb{K}[x,s,t,u,v]_d^{\mathbb{G}_a} = (\mathbb{K}[x,s,t,u,v]^{\mathbb{G}_a})_d$ is finitely generated.

Additionally, since $\mathfrak{f}_{R^D}$ is a radical ideal by \cite[\S 2.2]{Kem}, we have that $\sqrt{(\beta_0,\gamma_0, \delta_0)R^D} \subset \mathfrak{f}_{R^D}$. Now suppose that $f \in \mathfrak{f}_{R^D}$. Note that $R^D = B_0 + (\beta_n,\gamma_n,\delta_n)_{n\in \mathbb{N}} R^D$, so we may assume that $f \in B_0$. Since $(R^D)_f$ is finitely generated, we therefore have that $(R^D)_f = (B_N)_f$ for some $N \in \mathbb{N}$. Hence there is some $k > 0$ satisfying $f^k\beta_{N+1}, f^k\gamma_{N+1}, f^k \delta_{N+1} \in B_N$ and $f^k \in [B_N \colon B_{N+1}] \cap B_0 = (\beta_0, \gamma_0, \delta_0)B_0 \subset (\beta_0, \gamma_0, \delta_0)R^D$ by Lemma \ref{conductor}. Thus $\mathfrak{f}_{R^D} = \sqrt{(\beta_0,\gamma_0, \delta_0)R^D}$, completing the proof of the first statement of Theorem \ref{finite generation ideal}. 

It remains to show that $\text{L}(\mathcal{G}) \subset \text{L}(\mathfrak{f}_{R^D})$. By Proposition \ref{radical proof} we know that $\mathcal{G} = \lbrace \beta_i, \gamma_i, \delta_i \, \vert \, i \geq 0 \rbrace$ generates the finite generation ideal. Note that the leading monomials of these generators are the $b_n, c_n$ and $d_n$, $n \in \mathbb{N}$ described in the proof of Lemma \ref{SAGBI}. We have shown that applying the relations of these monomials to the corresponding generators yields an element with a leading monomial lying in $L(\mathcal{G})$. Additionally, we have shown that applying the relations between the generators corresponding to these leading monomials and $e$, the leading monomial of $g$, yields an element with leading monomial lying in $L(\mathcal{G})$. Any element in $L(\mathfrak{f}_{R^D})$ is obtained as the leading monomial of some combination of elements in $\mathfrak{f}_{R^D}$ and elements in $R^D$, which is generated by $\mathcal{G} \cup \lbrace g \rbrace$. Since all such combinations yield an element whose leading monomial lies in $L(\mathcal{G})$, we conclude that $\text{L}(\mathcal{G}) \subset \text{L}(\mathfrak{f}_{R^D})$.
\end{proof}